\documentclass{mcom-l}
\copyrightinfo{}{}
\makeatletter
\AtBeginDocument{\let\@serieslogo\@empty}
\makeatother

\usepackage{amssymb}
\usepackage{mathtools}
\usepackage{cases}
\usepackage{mathrsfs}
\usepackage{graphicx}
\usepackage{enumerate}
\usepackage{marginnote}
\usepackage[dvipsnames]{xcolor}
\usepackage{xstring}
\usepackage[hidelinks]{hyperref}
\allowdisplaybreaks[3]

\newtheorem{thm}{Theorem}[section]
\newtheorem{corollary}[thm]{Corollary}
\newtheorem{lem}[thm]{Lemma}
\newtheorem{prop}[thm]{Proposition}
\theoremstyle{definition}
\newtheorem{example}[thm]{Example}
\theoremstyle{remark}
\newtheorem{remarks}[thm]{Remark}

\numberwithin{equation}{section}

\newcommand{\blue}[1]{#1}
\newenvironment{bluechange}{}{}
\newcommand{\schemechange}[1]{#1}
\newenvironment{schemechanges}{}{}
\providecommand{\doi}[1]{}
\renewcommand{\doi}[1]{%
  \IfBeginWith{#1}{http}%
    {\href{#1}{\textcolor{black}{\nolinkurl{#1}}}}%
    {\href{https://doi.org/#1}{\textcolor{black}{\nolinkurl{#1}}}}}

\date{}

\newcommand{\nn}{\mathbb{N}}
\newcommand{\rr}{\mathbb{R}}

\def\LL{\mathcal L}

\def\<{\langle}
\def\>{\rangle}

\def\d"{^{\prime\prime}}

\begin{document}

\title[Tamed Euler approximation of invariant measures]{Tamed Euler Approximation of Invariant Measures for Weakly Interacting Particle Systems and McKean--Vlasov Equations}

\author[L. Cheng]{Lingyan Cheng}
\address{School of Mathematics and Statistics, Nanjing University of Science and Technology, Nanjing, Jiangsu 210094, PR China}
\email{cly@njust.edu.cn}

\author[R. Li]{Ruinan Li}
\address{School of Statistics and Data Science, Shanghai University of International Business and Economics, Shanghai 201620, PR China}
\email{ruinanli@amss.ac.cn}

\author[S. Liu]{Shiyu Liu}
\address{School of Mathematics and Statistics, Wuhan University, Wuhan, Hubei 430072, PR China}
\email{syliu\_math@whu.edu.cn}

\author[W. Liu]{Wei Liu}
\address{School of Mathematics and Statistics, Wuhan University, Wuhan, Hubei 430072, PR China}
\email{wliu.math@whu.edu.cn}
\thanks{Ruinan Li is the corresponding author.}

\subjclass[2020]{Primary 65C30; Secondary 60H35, 60K35}
\keywords{McKean--Vlasov equations, interacting particle systems, invariant measures, tamed Euler schemes, Wasserstein distance, reflection coupling}

\begin{abstract}

We analyze explicit tamed Euler approximations for weakly interacting
particle systems and McKean--Vlasov diffusions with additive noise and
polynomially weighted locally H\"older, possibly superlinear drift.
Under dissipativity at infinity and suitable stability assumptions, we
establish time-uniform $W_1$ bounds by reflection coupling, separating
relaxation, time-discretization, and finite-particle errors.
For small step sizes, we prove existence and uniqueness of numerical
invariant laws and bound their error relative to the continuous-time
particle invariant law uniformly in $N$ in a normalized joint $W_1$
distance. We further bound the error of the one-particle numerical
invariant marginal relative to the McKean--Vlasov invariant law by
$O(\theta^{\min\{\alpha,\beta/2\}}+N^{-1/2})$, where $\theta$ is the
step size, $N$ the particle number, $\alpha\in(0,1/2]$ the taming exponent,
and $\beta\in(0,1]$ the H\"older exponent. Analytical examples illustrate
the framework for smooth and non-Lipschitz drifts.

\end{abstract}

\maketitle
\hypersetup{
  pdftitle={Tamed Euler Approximation of Invariant Measures for Weakly Interacting Particle Systems and McKean--Vlasov Equations},
  pdfauthor={Lingyan Cheng, Ruinan Li, Shiyu Liu, Wei Liu},
  pdfsubject={Tamed Euler approximation of invariant measures for McKean--Vlasov equations}
}

\section{Introduction}
The approximation of invariant measures for stochastic differential equations (SDEs) is a classical and important topic in stochastic analysis, numerical probability, and sampling; see, for instance, \cite{BP23,BP24,DM17,FT21,PP23,Soize94}. McKean--Vlasov SDEs, also called distribution-dependent or mean-field SDEs, have attracted sustained attention because they arise naturally in stochastic control, mathematical finance, interacting particle systems, and machine learning; see, e.g., \cite{BLM17,CB18,HR21,JR15}. In this paper we study the long-time numerical approximation of invariant measures for the McKean--Vlasov SDE
\begin{equation}\label{MV}
	\begin{cases}
		dX_t=\sqrt{2}dB_t+b(X_t)dt+ G \circledast\mu_{t}(X_t)dt,\\
		X_0\sim\mu_0,
	\end{cases}
\end{equation}
where $\mu_t=\mathcal{L}(X_t)$ denotes the law of $X_t$ and
\[
G\circledast\mu_t(x):=\int_{\rr^d} G(x,y)\mu_t(dy).
\]
\blue{Here $B=(B_t)_{t\geq0}$ is a standard $d$-dimensional Brownian motion, and $X_0$, with law $\mu_0$, is independent of $B$. The precise assumptions on $b$ and $G$ are stated in Section~2.}

\begin{schemechanges}
At the level of probability laws, \eqref{MV} corresponds to a nonlinear
Fokker--Planck equation. In terms of a density $\mu_t(dx)=\rho_t(x)\,dx$,
it reads
\begin{equation}\label{eq:nonlinear-fokker-planck}
\partial_t\rho_t(x)=\Delta\rho_t(x)
-\nabla_x\cdot\left[\rho_t(x)\left(b(x)
+\int_{\rr^d}G(x,y)\rho_t(y)\,dy\right)\right],
\qquad t>0,\ x\in\rr^d.
\end{equation}
The equation is understood in the weak sense, with initial measure
$\mu_0$. Under the assumptions below, the invariant law $\pi^{MV}$
corresponds to a stationary probability solution of this equation.
\end{schemechanges}

When $G=0$, equation \eqref{MV} reduces to the Langevin diffusion
\begin{equation}\label{LD}
	\begin{cases}
		dX_t=\sqrt{2}dB_t+b(X_t)dt,\\
		X_0\sim \mu_0.
	\end{cases}
\end{equation}
For $b=-\nabla V$, standard well-posedness, normalizability, and
coercivity assumptions ensure that \eqref{LD} has an invariant density
proportional to $e^{-V}$ and converges to its invariant law $\pi$;
see \cite{P14}. Langevin dynamics therefore provides a continuous-time
model for sampling. Since exact simulation is generally unavailable,
one uses a time-discretization scheme. For globally Lipschitz drift,
the classical choice is the Euler discretization
\begin{equation*}
	\widetilde{X}_{(k+1)\theta}=\widetilde{X}_{k\theta}+ \theta b(\widetilde{X}_{k\theta})+\sqrt{2 \theta} \xi_{k+1},
\end{equation*}
where $\theta>0$ is the step size and $\left(\xi_{k}\right)_{k \geq 1}$ is an i.i.d. sequence of $d$-dimensional standard Gaussian random vectors. This iteration is commonly referred to as the Langevin Monte Carlo (LMC) algorithm, and its behavior has been studied extensively under various regularity and geometric assumptions; see, for example, \cite{BP23,BP24,DM17,FT21,PP23}. A well-known difficulty is that the explicit Euler scheme may diverge when the coefficients grow super-linearly \cite{HJK11}. This has motivated a large literature on numerical schemes for SDEs with nonglobally Lipschitz coefficients, including implicit methods \cite{HMS02,Hu96,Talay02} and explicit tamed schemes \cite{HJK12,Sabanis13}. \blue{The key idea of taming is to replace the original drift with a modified version having at most linear growth, leading to a scheme of the form}
\begin{equation}\label{eq:tamed-scheme}
	\widetilde{X}_{(k+1)\theta}=\widetilde{X}_{k\theta}+\theta b_\alpha(\widetilde{X}_{k\theta})+\sqrt{2\theta}\xi_{k+1},\,\,\widetilde{X}_0=X_0,
\end{equation}
\blue{Here $X_0$ is the initial condition in \eqref{LD}, and the tamed drift $b_\alpha$ depends on the step size $\theta$. It is chosen to have at most linear growth while converging to $b$ as $\theta\to0$; see \eqref{eq:tamed}. Compared with implicit schemes, tamed Euler methods retain the simplicity of an explicit scheme and remain computationally attractive in superlinear settings.}

\begin{schemechanges}
Under the polynomially weighted local H\"older condition with weight
exponent $l$ and H\"older exponent $\beta\in(0,1]$, we assume $l+\beta>1$
and use the taming power $\varkappa:=l+\beta-1$. This matches the growth
bound $|b(x)|=O(1+|x|^{\varkappa+1})$. Quadratic coercivity of the tamed
drift is imposed separately in {\bf (A2)}; it does not follow from the
H\"older condition alone. The definition and estimates are given in
Section~\ref{se:2}.
\end{schemechanges}

\begin{schemechanges}
Non-asymptotic Langevin approximation has been studied under different
regularity and geometric assumptions. In the strongly log-concave setting,
Dalalyan \cite{Dalalyan17} analyzes LMC in total variation using diffusion
ergodicity and a pathwise Kullback--Leibler estimate. Without convexity,
Mou, Flammarion, Wainwright and Bartlett \cite{MFWB22} obtain
finite-horizon KL discretization bounds with second-order step-size
dependence under globally Lipschitz drift, a Lipschitz Jacobian, and
additional moment and Fisher-information conditions; their sampling
corollary further assumes a log-Sobolev inequality. For superlinear
log-gradients, Lytras and Mertikopoulos \cite{LM25} develop tamed samplers
with KL, total-variation, and Wasserstein guarantees under local Lipschitz
regularity and Poincar\'e or log-Sobolev assumptions. Liu, Majka and
Monmarch\'{e} \cite{LMM23} establish $W_2$-contraction for Euler schemes
of elliptic diffusions and interacting particle systems under
contractivity at infinity, global Lipschitz regularity, sufficiently
large diffusivity, and weak interaction.
These results use different metrics and assumptions, so their rates are
not directly comparable. Here we seek direct long-time $W_1$ bounds for
a tamed scheme with polynomially weighted locally H\"older drift,
using distant dissipativity rather than functional inequalities or a
large-diffusivity condition.
\end{schemechanges}

For McKean--Vlasov SDEs the situation is more delicate, because one must also approximate the distribution-dependent interaction term. The standard implementable approximate model for \eqref{MV} is the interacting particle system
\begin{equation}\label{IPS}
	\begin{cases}
		dX_t^{i,N}=\sqrt{2}dB_t^i+b(X_t^{i,N})dt+\frac{1}{N-1}\sum\limits_{j\neq i}G(X_t^{i,N},X_t^{j,N})dt,\\
		X_0^{i,N}=X_0^i,i=1,...,N,
	\end{cases}
\end{equation}
Here $X_0^1,\ldots,X_0^N$ are i.i.d. with law $\mu_0$, and
$B^1,\ldots,B^N$ are mutually independent $d$-dimensional Brownian
motions, independent of the entire initial family. Under suitable
assumptions, the law of a fixed particle converges to the law of
\eqref{MV} as $N\to\infty$, the propagation of chaos phenomenon;
see \cite{DEGZ20,GLWZ22,HRZ24,LWZ21,M01,Meleard96,Sznitman91}.
An implementable algorithm further requires time discretization of
the particle system.

Euler schemes for McKean--Vlasov SDEs with linear growth have been studied in \cite{AK02,BH22,BT97}. For superlinear coefficients, both implicit schemes \cite{CR22,DES22} and tamed Euler schemes \cite{DES22,LSW23} have been developed. \schemechange{These studies primarily address finite-time approximation. Here we focus on quantitative approximation of the invariant measure $\pi^{MV}$ of \eqref{MV} by a tamed Euler discretization of the interacting particle system.} More precisely, for each $1\leq i\leq N$ and $t\in[k\theta,(k+1)\theta]$, $k \in \nn
$, we consider the continuous-time interpolation
\begin{equation*}
	\widetilde{X}_{t}^{i,N}=\widetilde{X}_{k\theta}^{i,N}+\int_{k\theta}^{t}b_\alpha(\widetilde{X}_{k\theta}^{i,N})ds+\frac{1}{N-1}\int_{k\theta}^{t}\sum\limits_{j\neq i}G(\widetilde{X}_{k\theta}^{i,N},\widetilde{X}_{k\theta}^{j,N})ds+\sqrt{2}\int_{k\theta}^{t}dB_s^i,
\end{equation*}
where $\widetilde{X}_0^{i,N}=X^i_0$ for any $1\leq i\leq N$. We study the convergence of $\widetilde{X}_{t}^{i,N}$ toward $\pi^{MV}$ in the $W_1$-Wasserstein distance.
The resulting particle method provides a mesh-free probabilistic
approximation of \eqref{eq:nonlinear-fokker-planck}, whose long-time
accuracy is quantified here in $W_1$.

\begin{schemechanges}
Several approaches to long-time McKean--Vlasov approximation are already
available, including self-interacting methods \cite{CD25,CL24,DJL23},
decreasing-step ergodic simulation \cite{CP24}, and unbiased multilevel
estimation \cite{ACJ24}. For explicit particle schemes,
\cite[Theorem~5.16]{CLLW24} bounds the bias of a one-particle numerical
invariant marginal under local Lipschitz state regularity and global
$W_2$-contractivity for the long-time result. The fully superlinear
framework of \cite{SNKR25} allows both drift and diffusion to depend
superlinearly on state and measure variables, and establishes strong
propagation of chaos, strong $L^p$ convergence of an explicit tamed
scheme, and $W_2$-ergodicity under its contractivity assumptions.

Bao and Hao \cite{BH25} establish uniform-in-time $W_1$ estimates under
dissipativity at infinity, with applications to several Euler-type
schemes. Their specific tamed Euler result
\cite[Theorem~1.11]{BH25} uses gradient-based taming and weighted
Lipschitz regularity. Our analysis treats a different tamed scheme under
weighted local H\"older regularity and provides numerical invariant-law
bounds uniform in $N$ in the normalized joint particle distance.
\end{schemechanges}

\begin{schemechanges}
Our coupling argument builds on the approximate componentwise reflection
coupling and adapted distance developed in \cite{Eberle16,LWZ21}. We use
this established framework for numerical perturbation analysis rather
than propose a new coupling principle. Whereas \cite{LWZ21} studies
continuous-time ergodicity and propagation of chaos, we compare the exact
dynamics with their tamed discretization, using uniform moment estimates
and quantitative bounds on the taming and time-discretization errors.
Our setting has additive noise and weak interaction of the form
$b(x)+G\circledast\mu(x)$. Weighted local H\"older regularity, distant
two-point dissipativity, and the upper bound on the dissipativity profile
are combined with the tamed coercivity condition {\bf (A2)} and the stated
interaction and step-size restrictions. These assumptions permit local
expansion, without requiring a log-Sobolev inequality or global
$W_2$-contraction; see Example~\ref{ex:holder-local-expansion}.
\end{schemechanges}

The main contributions of the paper can be summarized as follows.
\begin{schemechanges}
\begin{enumerate}
\item We establish time-uniform $W_1$ error bounds for tamed Euler
approximations of Langevin diffusions and weakly interacting particle
systems with polynomially weighted locally H\"older, possibly superlinear
drift. The analysis permits local expansion of the drift under distant
two-point dissipativity and the stated assumptions. In particular, the
coercivity condition {\bf (A2)} on the tamed drift, together with the
growth bounds and the restrictions on the interaction and step size,
yields uniform moment estimates and a Lyapunov drift inequality for the
numerical chains. For the particle system, the discretization bound is
uniform in $N$ in the normalized distance $\bar d_N=N^{-1}d_{l^1}$.
The error coefficients are explicitly displayed or recursively computable
from the model parameters and initial moments.

\item For sufficiently small step sizes, we prove existence and
uniqueness of invariant laws for the Langevin and $N$-particle tamed
Euler chains. We obtain quantitative $W_1$ bias bounds relative to the
corresponding continuous-time invariant laws, with discretization order
$O(\theta^{\min\{\alpha,\beta/2\}})$. For the particle invariant law,
this order holds uniformly in $N$ in the normalized distance $\bar d_N$.

\item Combining the numerical perturbation estimate with
continuous-time contraction and uniform-in-time propagation of chaos,
we obtain a direct approximation bound for the McKean--Vlasov invariant
law, separating relaxation, discretization, and finite-particle errors.
In particular, the one-particle marginal of the numerical invariant law
approximates $\pi^{MV}$ with error
$O(\theta^{\min\{\alpha,\beta/2\}}+(N-1)^{-1/2})$.
As consequences, we bound the bias of the expected time-averaged
empirical measure and, under the stated moment assumptions, the expected
$W_1$ error of the fixed-time empirical particle measure, retaining the
dimension-dependent sampling term in the latter.
\end{enumerate}
\end{schemechanges}

Although the McKean--Vlasov setting is our main target, the case $G=0$ remains important both conceptually and technically. We therefore present the single-particle Langevin case first. It already contains the essential coupling argument in a simpler form and clarifies the structure of the later proofs for the particle system and the limiting equation.

The remainder of the paper is organized as follows. We introduce the assumptions, notations, and main results in Section 2. In Section 3 we deal with the Langevin diffusion case. Section 4 is devoted to the interacting particle system and to the McKean--Vlasov equation. The appendix contains the moment estimates used in the main proofs.

\section{Preliminaries}\label{se:2}

\subsection{Notations and assumptions}

\blue{Let $(\Omega,\mathscr F,(\mathscr F_t)_{t\ge0},\mathbb P)$ be a filtered probability space satisfying the usual conditions. 
The initial random variables are $\mathscr F_0$-measurable, and the Brownian motions are adapted to $(\mathscr F_t)_{t\ge0}$ and independent of the initial data.} For $a=(a_1,\ldots,a_d)$ and $b=(b_1,\ldots,b_d)$ in $\mathbb R^d$, let $|a|=(\sum_{i=1}^d a_i^2)^{1/2}$ and $\langle a,b\rangle=\sum_{i=1}^d a_ib_i$. For $r\in\mathbb R$, set $r^+:=\max\{r,0\}$.
For any bounded continuous function $f$ on $\mathbb{R}^{d}$, $\|f\|_\infty=\sup_{x\in\mathbb{R}^{d}}|f(x)|$ is the sup-norm.
For any function $G: \mathbb{R}^{d} \times \mathbb{R}^{d} \rightarrow \mathbb{R}^d$,
$\nabla_{x} G$ and $\nabla_{y} G$ denote the Jacobian matrices of $G$ with respect to  the first and second variable, respectively. Define  $\left\|\nabla_{x} G\right\|_{\infty}:=\sup _{x, y \in \mathbb{R}^{d}} \sup _{z \in \mathbb{R}^{d},|z|=1}\left|\nabla_{x} G(x, y) z\right|$, and $\left\|\nabla_{y} G\right\|_{\infty}$ is defined analogously.

On the product space $(\mathbb R^d)^N$, we use the $l^1$ metric (generalized Hamming metric)
\begin{equation*}
	d_{l^1}(x,y)=\sum_{i=1}^N|x^i-y^i|,\qquad x,y\in(\mathbb R^d)^N.
\end{equation*}
For any two probability measures $\mu,\nu$ on $(\mathbb{R}^d)^N,$
we consider the $1$-Wasserstein distance associated with the metric $d_{l^1}$, namely
\begin{equation*}
	W_{1,d_{l^1}}(\mu,\nu):=\inf\limits_{P\in\Pi(\mu,\nu)}\int\int_{(\mathbb{R}^d)^N\times(\mathbb{R}^d)^N}d_{l^1}(x,y)P(dx,dy),
\end{equation*}
where $\Pi(\mu,\nu)$ denotes the set of all couplings of $\mu$ and $\nu$, that is, the set of all probability measures on $\left(\mathbb{R}^{d}\right)^{N} \times\left(\mathbb{R}^{d}\right)^{N}$ whose marginal distributions in the $x$- and $y$-variables are $\mu$ and $\nu$, respectively.
For $N=1$, we simply write $W_1$; for $N\ge2$, the notations $W_{1,d_{l^1}}$ specifies the product-space metric.

 \blue{Throughout this paper we impose the following assumptions on the drift $b$, its tamed version $b_\alpha$, and the dissipativity profile introduced below:}
\begin{enumerate}
	\item[{\bf (A1)}] $b$ satisfies a polynomially weighted local H\"older condition: there exist $L>0$, $l\geq0$, and $\beta\in(0,1]$ with $l+\beta>1$ such that, for all $x,y\in\mathbb{R}^d$,
	\begin{equation}\label{A1}	|b(x)-b(y)|\leq L(1+|x|^l+|y|^l)|x-y|^\beta.\end{equation}
	
	\item[{\bf (A2)}] With $l$ and $\beta$ as in Assumption (A1), set $\varkappa:=l+\beta-1>0$. For any step size $\theta>0$ and a fixed $\alpha\in(0,\frac{1}{2}]$, define
    \begin{equation}\label{eq:tamed}
        b_{\alpha}(x):=\frac{b(x)}{1+\theta^\alpha |x|^\varkappa}.
    \end{equation}
   Assume that there exist positive constants $c_1$, $c_2$, and $\theta^*$ such that, for all $x\in\mathbb{R}^d$ and $\theta\in(0,\theta^*)$,
	\begin{equation}\label{A11}
		\langle x,b_{\alpha}(x)\rangle\leq -c_1|x|^2+c_2.
	\end{equation} 
\end{enumerate}

\begin{remarks}
 Let $K:=|b(0)|+2L$. Under Assumption {\bf (A1)} and the definition \eqref{eq:tamed}, for every $x\in\rr^d$ and $0<\theta\le1$,
		\begin{equation}\label{prop1}
		\begin{aligned}
			&|b(x)|\le K(1+|x|^{\varkappa+1});\\
			&|b_\alpha(x)|\le 2K\theta^{-\alpha}(1+|x|);\\
			&|b(x)-b_\alpha(x)|\le 2K\theta^\alpha(1+|x|^{2\varkappa+1}).
		\end{aligned}
		\end{equation}
Indeed, taking $y=0$ in \eqref{A1} and using the zeroth-power convention when $l=0$ gives
\[
\begin{aligned}
|b(x)|&\le |b(0)|+L(1+|x|^l)|x|^\beta\\
&\le |b(0)|+L(1+2|x|^{l+\beta})
&\le K(1+|x|^{l+\beta})=K(1+|x|^{\varkappa+1}).
\end{aligned}
\]
The factor $2L$ covers the endpoint $l=0$, where the weight in \eqref{A1} equals $2$. Since $\varkappa\ge0$, splitting into $|x|\le1$ and $|x|\ge1$ yields
\[
|b_\alpha(x)|\le K\frac{1+|x|^{\varkappa+1}}{1+\theta^\alpha|x|^\varkappa}
\le 2K\theta^{-\alpha}(1+|x|).
\]
Finally,
\[
\begin{aligned}
|b(x)-b_\alpha(x)|
&=|b(x)|\frac{\theta^\alpha|x|^\varkappa}{1+\theta^\alpha|x|^\varkappa}\\
&\le K\theta^\alpha\bigl(|x|^\varkappa+|x|^{2\varkappa+1}\bigr)
\le 2K\theta^\alpha(1+|x|^{2\varkappa+1}).
\end{aligned}
\]
Here the taming error requires moment order $2\varkappa+1=2l+2\beta-1$. The weight exponent $l$ in the local H\"older estimate \eqref{A1} is unchanged. Quadratic dissipativity \eqref{A11} remains an additional assumption on the drift defined by \eqref{eq:tamed}; it does not follow from \eqref{A1} alone.
\end{remarks}
\begin{remarks}{
			Since $b_\alpha(x)\to b(x)$ pointwise as $\theta\downarrow0$, Assumption {\bf (A2)} also gives the dissipativity of the original drift:
			\[
			\langle x,b(x)\rangle\leq -c_1|x|^2+c_2,\qquad x\in\mathbb R^d.
			\]
			This observation will be used for the continuous-time moment estimates for the McKean--Vlasov equation \eqref{MV}.
		}
	\end{remarks}

\begin{remarks}\label{rem:A2-criterion}
A sufficient condition for Assumption {\bf (A2)} can be stated directly in terms of the original drift. Suppose that, for some $a>0$ and $B\ge0$,
\[
\langle x,b(x)\rangle\le -a|x|^{\varkappa+2}+B(1+|x|^2),
\qquad x\in\mathbb R^d.
\]
Then, for every  $c_1>0$, there exist $c_2,\theta^*>0$ such that \eqref{A11} holds for all $x\in\mathbb R^d$ and $0<\theta<\theta^*$. Indeed, put $s=\theta^\alpha$, choose $R\ge1$ with $R^\varkappa\ge 2(B+c_1)/a$, and require $s\le a/(2c_1)$. For $|x|\ge R$,
\[
\langle x,b(x)\rangle
\le -c_1(1+s|x|^\varkappa)|x|^2+B,
\]
so division by $1+s|x|^\varkappa$ gives \eqref{A11} with $c_2\ge B$. For $|x|\le R$, the same inequality holds with $c_2\ge B(1+R^2)+c_1R^2$. Thus $c_1$ may be chosen large enough to meet the additional condition {\bf (A2')} by reducing $\theta^*$ if necessary.
\end{remarks}

\begin{enumerate}
	\item[{\bf(A2')}]\label{A21} {When $\alpha=\frac{1}{2}$, assume additionally that $c_1>4K^2$.}
\end{enumerate}

\blue{We next introduce a {\it \bf dissipativity rate} $\beta_0(r)$ for the drift of \eqref{MV} at distance $r>0$. It is a continuous function on $(0,+\infty)$ satisfying}
\begin{equation}\label{dissipativity}
	\langle x-y,b(x)-b(y) +G(x,z)-G(y,z) \rangle\leq \beta_0(r)|x-y|
\end{equation}
for any $x,y,z\in\mathbb{R}^d$ with $|x-y|=r.$
Notice that for Langevin diffusion \eqref{LD}, this reduces to 
\begin{equation}\label{dissipativity-single}
	\langle x-y,b(x)-b(y) \rangle\leq \beta_0(r)|x-y|.
\end{equation}
Assume that $\beta_0(r)$ satisfies the following conditions:

\begin{enumerate}
	\item[{\bf(A3)}] \begin{equation}\label{A3-infinity} \limsup_{r\to +\infty} \frac{\beta_0(r)}{r} < 0,
	\end{equation}
	i.e., the drift  is dissipative at infinity;
	\item[{\bf(A4)}] there exists $M\in\mathbb{R},$ such that for any $r>0,$
	\begin{equation}\label{A4-linear}
		\beta_0(r)\leq rM.
	\end{equation}
\end{enumerate}

Without loss of generality, we assume that $\beta_0(0)=0$. This extension is continuous in the settings below: with $G=0$, or under (A5), taking $x=ru$, $y=z=0$, and $|u|=1$ in the dissipativity inequality gives
\[
-C(r^\beta+r)\le\beta_0(r)\le Mr,\qquad 0<r\le1,
\]
for a finite constant $C$, by (A1) and the local Lipschitz bound on $G$. Hence $\beta_0(r)\to0$ as $r\downarrow0$. This convention is used whenever a coupled distance vanishes.

We adopt the reference function $h$ used in \cite{LWZ21} for the coupling construction. 
For any function $f\in C^2(0,+\infty)$ and $r>0$, let $\mathcal{L}_{ref}$ be the generator defined by
\begin{equation*}
	\mathcal{L}_{ref}f(r):=4f''(r)+\beta_0(r)f'(r).
\end{equation*}

Let $h:\rr^+\to \rr^+$ be the function determined by $h(0)=0$ and
\begin{equation}\label{h'}
	h^{\prime}(r)=\frac{1}{4} \exp\left(-\frac{1}{4} \int_0^r \beta_0(s)ds\right) \int_r^{+\infty} s\cdot\exp\left(\frac{1}{4} \int_0^s \beta_0(u)du\right) ds.
\end{equation}
As in \cite{LWZ21}, the continuity of $\beta_0$ and its dissipativity at infinity imply that $h'$ is $C^1$ and that $h$ is a well-defined $C^2$ function. Moreover, $h$ is a solution of the one-dimensional Poisson equation
\begin{equation}\label{h}
	\LL_{ref} h(r)=4 h^{\prime\prime}(r)+\beta_0(r)h^{\prime}(r)=-r,\ r>0
\end{equation}
with $h(0)=0$. 
Note that when the dissipativity at infinity condition {\bf (A3)} is satisfied, $\beta_0(r)$ can be bounded from above by $-a_0r+b_0$ (with $a_0,b_0>0$), so $\|h'\|_\infty:=\sup_{r\ge0}|h'(r)|<+\infty$. In particular, $h(r)>0$ for $r>0.$ 

\begin{remarks}[Non-Lipschitz drifts]\label{rem:holder-drift}
\rm The assumptions allow non-Lipschitz drifts for every
$\beta\in(0,1)$. For instance, take $d=1$, $G=0$, and
\[
b(x)=-x-x^3-\operatorname{sgn}(x)|x|^\beta,
\qquad \operatorname{sgn}(0)=0.
\]
The signed-power H\"older estimate verifies {\bf (A1)} with
$l=3-\beta$, $L=8$, $\varkappa=2$, and $K=16$.
Monotonicity gives $(x-y)(b(x)-b(y))\le-|x-y|^2$, so
{\bf (A3)--(A4)} hold with $\beta_0(r)=-r$ and $M=-1$.
For any $c_1>0$, {\bf (A2)} holds with
$c_2=1+c_1^2/2$ and $\theta^*=(2c_1)^{-1/\alpha}$;
choosing $c_1>4K^2$ also ensures {\bf (A2')} when $\alpha=1/2$.
A finite initial sixth moment suffices for the results below.
A locally expanding interacting model is presented in
Example~\ref{ex:holder-local-expansion}.
\end{remarks}

\blue{We next present the settings and main results separately for the single-particle Langevin diffusion, the interacting particle system, and the McKean--Vlasov equation. A common feature of the three statements below is that the total approximation error is decomposed into explicit contributions. In the single-particle and particle settings, the error splits into relaxation and discretization terms; in the McKean--Vlasov setting, an additional finite-particle term appears.}
\subsection{Langevin diffusion}
\blue{First, we consider the case $G=0$ in the McKean--Vlasov equation, which gives the following classical Langevin diffusion:}

\begin{equation}\label{singleparticle}
	\begin{cases}
		dX_t=\sqrt{2}dB_t+b(X_t)dt,\\
		X_0\sim \mu_0.
	\end{cases}
\end{equation}
\blue{Under Assumptions {\bf (A1)},  {\bf(A3)} and  {\bf(A4)}, 
the SDE above admits a unique nonexplosive strong solution and a unique invariant measure, denoted by $\pi$.}

Consider the following Euler-type discretization of SDE \eqref{singleparticle}:
\begin{equation}\label{discretization} 
	\begin{cases}
		\widetilde{X}_{(k+1)\theta}=\widetilde{X}_{k\theta}+\theta b_{\alpha}(\widetilde{X}_{k\theta})
		+\sqrt{2}(B_{(k+1)\theta}-B_{k\theta}),\\
		\widetilde{X}_0=X_0,
\end{cases}\end{equation}
where  $\theta>0$ is the step size and $k=0, 1, \cdots$.

The interpolation process is defined as follows:
\begin{equation}\label{interpolation-single}
	\widetilde{X}_{t}=\widetilde{X}_{k\theta}+(t-k\theta)b_{\alpha}(\widetilde{X}_{k\theta})+\sqrt{2}(B_{t}-B_{k\theta}), \ \ t\in[k\theta,(k+1)\theta].
\end{equation}	

For any $r>0$, denote by $m_r:=\int_{\mathbb R^d}|x|^r\,\mu_0(dx)$ the $r$-th moment of $\mu_0$. Let $\mu_{\widetilde{X}_t}$ be the law of $\widetilde{X}_t$. 

\begin{bluechange}
Set
\begin{align*}
	\eta:=\begin{cases}
		\min\left\{\frac{1}{4},\theta^*,\frac{1}{c_1},(\frac{c_1}{8K^2})^{\frac{1}{1-2\alpha}}\right\}, &\alpha\in(0,\frac{1}{2}),\\
		\min\left\{\frac{1}{4},\theta^*,[2(c_1-4K^2)]^{-1}\right\}, &\alpha=\frac{1}{2},
	\end{cases}
\end{align*}
and, for $r>0$, define the moment constants
\begin{equation}\label{moment-inter-Langevin}
	C_r:=\sup_{0<\theta<\eta}\sup_{t\ge0}\mathbb E|\widetilde X_t|^r,
	\qquad C_0:=1.
\end{equation}
Their finiteness up to the moment order required below is established in Lemma~\ref{pri}.
\end{bluechange}

We have the following result.

\begin{thm}\label{mainre1}

	Suppose that Assumptions {\bf (A1)--(A4)} hold and, when $\alpha=\frac12$, assume additionally {\bf (A2')}. Assume that $m_q<\infty$ for some even integer $q$ satisfying $q\geq\max\{2,2l+2\beta-1\}$. Then, for any $0<\theta<\eta$, any $\varepsilon>0$ satisfying $1-\varepsilon M>0$, and any $T>0$,
	\begin{equation}\label{eq:Lan-main-bound}
		W_1(\mu_{\widetilde{X}_T},\pi)\leq A_\varepsilon e^{-K_\varepsilon T}W_1(\mu_0,\pi)
		+\sup\limits_{r>0}\frac{r}{h(r)+\varepsilon r}\frac{\Gamma_1\theta^\alpha+\Gamma_2\theta^{\frac\beta2}}{K_\varepsilon}.
	\end{equation}
	The constants in \eqref{eq:Lan-main-bound} are
	\begin{align*}
		A_\varepsilon&=\sup\limits_{r>0}\frac{r}{h(r)+\varepsilon r}\sup\limits_{r>0}\frac{h(r)+\varepsilon r}{r},\\
		K_\varepsilon&=\frac{1-\varepsilon M}{\|h'\|_\infty+\varepsilon},\\
		\Gamma_1&=2K(\|h'\|_\infty+\varepsilon)(1+C_{2l+2\beta-1}),\\
		\Gamma_2&=L(\|h'\|_\infty+\varepsilon)(1+2(C_{\frac{2l}{2-\beta}})^{1-\frac{\beta}{2}})(8K^2(1+C_2)+2d)^{\frac{\beta}{2}}.
	\end{align*}

\end{thm}

For later reference, denote the two terms on the right-hand side of \eqref{eq:Lan-main-bound} by
\begin{equation}\label{eq:relax}
\begin{split}
    \mathcal E_{\rm relax}^{\rm Lan}(T)&:=A_\varepsilon e^{-K_\varepsilon T}W_1(\mu_0,\pi),\\
	\mathcal E_{\rm disc}^{\rm Lan}(\theta)&:=\sup\limits_{r>0}\frac{r}{h(r)+\varepsilon r}
	\frac{\Gamma_1\theta^\alpha+\Gamma_2\theta^{\frac\beta2}}{K_\varepsilon}.
\end{split}
\end{equation}
 Here, since $h(r)>0$ and $h(r)\leq\|h'\|_\infty r$ for $r>0$, it follows that $1\leq A_\varepsilon\leq(\|h'\|_\infty+\varepsilon)/\varepsilon$. Moreover,
 $q\ge 2\varkappa+1=2l+2\beta-1\ge 2l/(2-\beta)$ implies $C_{2\varkappa+1}$ and $C_{2l/(2-\beta)}$ are finite by Lemma~\ref{pri}.  This moment comparison, which precedes Theorem~\ref{mainre2}, also ensures that \(\widetilde C_{2\varkappa+1}\) and \(\widetilde C_{2l/(2-\beta)}\) are finite under the stated initial moment condition.

\begin{remarks}\label{rem:dim1}
    The two terms of \eqref{eq:relax} represent the relaxation and discretization errors, respectively; the two powers of $\theta$ in the latter correspond to the taming error and the local H\"older error.  Thus the discretization order remains $O(\theta^{\min{\{\alpha,\beta/2\}}})$.  

    Regarding dimension dependence, it follows from Lemma~\ref{pri} that \(C_r\lesssim d^{r/2}\), with the remaining model and coupling constants independent of \(d\). The discretization bound in Theorem~\ref{mainre1} then has dimension dependence of order \(O(d^{\max\left\{l+\beta-\frac12,\frac{\beta+l}{2}\right\}})\). In the linear growth case \(\beta+l\le 1\), taming is unnecessary, and the same argument applied to the standard Euler scheme yields a dimension factor of order \(\sqrt d\). For comparison, the leading KL bound in \cite{MFWB22} has a dimension factor of order \(d^2\).

\end{remarks}

\begin{thm}\label{thm:mainre2}
	Assume the same assumptions and notations as in Theorem \ref{mainre1}. Then, when $0<\theta<\eta,$  $\left\{\widetilde{X}_{k\theta}\right\}_{k\in\mathbb{N}}$ admits a unique invariant distribution denoted by $\pi^{\rm TEM}_\theta$ and 

	\begin{equation}\label{eq:Lan-invariant-bound}
		W_1(\pi^{\rm TEM}_\theta,\pi)\le \mathcal E_{\rm disc}^{\rm Lan}(\theta).
	\end{equation}

\end{thm}

\subsection{Interacting particle system}
Consider the following $N$-particle system:
\begin{equation}\label{multiparticle}
	\begin{cases}		dX_t^{i,N}=\sqrt{2}dB_t^i+b(X_t^{i,N})dt+\frac{1}{N-1}\sum\limits_{j\neq i}G(X_t^{i,N},X_t^{j,N})dt,\\
		X_0^{i,N}=X_0^i,i=1,...,N,
	\end{cases}
\end{equation}
\blue{where $X_0^1,\ldots,X_0^N$ are i.i.d. random variables with common law $\mu_0$, and $B^1,\ldots,B^N$ are independent $d$-dimensional Brownian motions in $\mathbb{R}^d$, independent of the entire initial family $(X_0^1,\ldots,X_0^N)$.}

\blue{We introduce the following key assumption on the interaction drift:}

\[
\text{({\bf H})}:\qquad \|\nabla_{y}G\|_{\infty}\,\|h^{\prime}\|_{\infty}<1.
\]
Here $h$ is given by \eqref{h'}.

Furthermore, we assume that 
\begin{align}	\label{A5}
	\text{(\bf A5)}:\quad
	\blue{G\in C^1(\mathbb{R}^d\times\mathbb{R}^d;\mathbb{R}^d),
	\qquad \|\nabla_xG\|_\infty+\|\nabla_yG\|_\infty<+\infty.}
\end{align}

Denote
\[
K_G:=\max\left\{|G(0,0)|,\|\nabla_xG\|_\infty,\|\nabla_yG\|_\infty\right\}.
\]
\blue{Then, under Assumptions {\bf (H)} and {\bf (A5)}, we have}
\begin{equation}\label{Glip}
	|G(x,y)|\leq {K_G}(1+|x|+|y|),\blue{\qquad \forall x,y\in\mathbb{R}^d.}
\end{equation}

\begin{remarks}\label{rem:LWZ-general-drift}
 Although \cite{LWZ21} is formulated for a gradient system, its Remark~2.13 extends the relevant results to a general drift that is uniformly dissipative in the state variable and sufficiently weakly Lipschitz in the law. In the present setting, let
\[
\mathfrak b(x,\mu):=b(x)+\int_{\mathbb R^d}G(x,z)\,\mu(dz).
\]
Under Assumptions {\bf (A3)--(A5)}, for $r=|x-y|$ we have
\begin{align}
\langle x-y,\mathfrak b(x,\mu)-\mathfrak b(y,\mu)\rangle
&\le \beta_0(r)r,\\
|\mathfrak b(x,\mu)-\mathfrak b(x,\nu)|
&\le \|\nabla_yG\|_\infty W_1(\mu,\nu).
\end{align}
Indeed, the first inequality follows by integrating \eqref{dissipativity} with respect to $\mu$, and the second follows by coupling $\mu$ and $\nu$ and using the $y$-Lipschitz bound on $G$. Hence Assumption {\bf (H)} is precisely the required weak-interaction condition. Consequently as indicated in  \cite{LWZ21}, the existence and uniqueness of the solutions to the McKean--Vlasov equation  and interacting particle system hold. The contraction estimate of \cite[Corollary~2.9]{LWZ21} and the uniform-in-time propagation-of-chaos estimate of \cite[Theorem~2.11(b)]{LWZ21} apply with $\|\nabla_{xy}^2W\|_\infty$ replaced by $\|\nabla_yG\|_\infty$. For concentration, we use the coupling and martingale strategy of \cite[Theorem~2.5]{LWZ21}, but derive the constants directly in the proof of Corollary~\ref{coro:concen}, retaining the Brownian variance factor $2$ from the noise $\sqrt2\,dB_t$. The Lyapunov and contraction estimates also give unique invariant laws for \eqref{multiparticle} and \eqref{MV}, denoted by $\pi^N$ and $\pi^{MV}$, respectively.
\end{remarks}

If $K_G=0$, then $G(0,0)=0$ and both Jacobians of $G$ vanish identically; hence $G\equiv0$ on $\mathbb R^d\times\mathbb R^d$. In that case there is no interaction and the model reduces to independent copies of the Langevin dynamics considered in Section~2.2. Therefore, throughout the genuinely interacting case below, we assume $K_G>0$.

\begin{enumerate}
	\item[{\bf(A5')}]\label{A51} {When $\alpha=\frac{1}{2}$, assume additionally that $c_1-\|\nabla_xG\|_\infty-\sqrt{2}K_G-8K^2>0$, where $c_1$ and $K$ are the constants given in {\bf (A2)} and Remark 2.1.} 
\end{enumerate}

We next consider an Euler-type discretization of the interacting particle system. Let $\widetilde X_0^{i,N}=X_0^i$, $1\le i\le N$, be the initial values. For $k\in\mathbb N$ and $1\le i\le N$, the scheme is defined by
\begin{equation}\label{chazhi}
	\widetilde{X}_{(k+1)\theta}^{i,N}=\widetilde{X}_{k\theta}^{i,N}+\theta b_{\alpha}(\widetilde{X}_{k\theta}^{i,N})+\frac{\theta}{N-1}\sum\limits_{j\neq i}G(\widetilde{X}_{k\theta}^{i,N},\widetilde{X}_{k\theta}^{j,N})+\sqrt{2}(B^i_{(k+1)\theta}-B^i_{k\theta}).
\end{equation}
For $t\in[k\theta,(k+1)\theta]$, define the interpolation by
\begin{equation}\label{interpolation}
\begin{aligned}
	\widetilde{X}_{t}^{i,N}&=\widetilde{X}_{k\theta}^{i,N}+(t-k\theta)b_{\alpha}(\widetilde{X}_{k\theta}^{i,N})\\
		&\quad+\frac{t-k\theta}{N-1}\sum\limits_{j\neq i}G(\widetilde{X}_{k\theta}^{i,N},\widetilde{X}_{k\theta}^{j,N})+\sqrt{2}(B^i_{t}-B^i_{k\theta}).
\end{aligned}
\end{equation}

Let $\mu_{\widetilde{X}^N_t}$ be the law of $\widetilde{X}^N_t=(\widetilde{X}^{1,N}_t,\cdots,\widetilde{X}^{N,N}_t)$.

\begin{bluechange}
Set
\begin{align*}
	\widetilde{\eta}:=\begin{cases}
		\min\left\{\begin{gathered}\frac{1}{4},\theta^*,(\frac{c_1-\|\nabla_xG\|_\infty -\sqrt{2}K_G}{16K^2 +12K_G^2})^{\frac{1}{1-2\alpha}},\\
		\frac{1}{2c_1-2\|\nabla_xG\|_\infty-\sqrt{2}K_G}\end{gathered}\right\}, &\alpha\in(0,\frac{1}{2}),\\
		\min\left\{\begin{gathered}\frac{1}{4},\theta^*,\frac{c_1-\|\nabla_xG\|_\infty -\sqrt{2}K_G-8K^2}{12K_G^2},\\
		\frac{1}{c_1-\|\nabla_xG\|_\infty-\sqrt{2}K_G-8K^2},\\
		\frac{1}{2c_1-2\|\nabla_xG\|_\infty-\sqrt{2}K_G}\end{gathered}\right\}, &\alpha=\frac{1}{2}.
	\end{cases}
\end{align*}
For $r>0$, define the moment constants
\begin{equation}\label{moment-inter-Npartical}
\widetilde C_r:=\sup_{N\ge2}\sup_{0<\theta<\widetilde\eta}\sup_{t\ge0}
	\mathbb E|\widetilde X_t^{1,N}|^r,
	\qquad \widetilde C_0:=1.
\end{equation}
Their finiteness up to the moment order required below is established in Lemma~\ref{primul}.
\end{bluechange}

\begin{thm}\label{mainre2}

	Suppose that Assumptions {\bf (A1)--(A5)} and {\bf (H)} hold and, when $\alpha=\frac12$, assume additionally {\bf (A5')}. Assume that $m_q<\infty$ for some even integer $q$ satisfying $q\geq\max\{2,2l+2\beta-1\}$, and that $c_1-\|\nabla_{x}G\|_\infty-\sqrt{2}K_G>0$. Then, for any $0<\theta<\widetilde{\eta}$, any $\varepsilon>0$ satisfying $1-\varepsilon M-(\|h'\|_\infty+\varepsilon)\|\nabla_{y}G\|_{\infty}>0$, and any $T>0$,
	\begin{equation}\label{eq:N-main-bound}
	\begin{aligned}
		W_{1,d_{l^1}}(\mu_{\widetilde{X}^N_T},\pi^N)
		\leq{}&A_\varepsilon e^{-\widetilde K_\varepsilon T}W_{1,d_{l^1}}(\mu_0^{\otimes N},\pi^N)\\
		&+\sup\limits_{r>0}\frac{r}{h(r)+\varepsilon r}
		\frac{\Gamma_3\theta^\alpha+\Gamma_4\theta^{\frac\beta2}+\Gamma_5\theta^{\frac12}}{\widetilde K_\varepsilon}N.
	\end{aligned}
	\end{equation}
	The constants in \eqref{eq:N-main-bound} are
	\begin{align*}
		A_\varepsilon&=\sup_{r>0}\frac{r}{h(r)+\varepsilon r}\sup_{r>0}\frac{h(r)+\varepsilon r}{r},\\
		\widetilde{K}_\varepsilon&=\frac{1-\varepsilon M-(\|h'\|_\infty+\varepsilon)\|\nabla_{y}G\|_{\infty}}{\|h'\|_\infty+\varepsilon},\\
		\Gamma_{3}&=(\|h'\|_\infty+\varepsilon)(2K+2K\widetilde{C}_{2l+2\beta-1}),\\
		\Gamma_4&=L(\|h'\|_\infty+\varepsilon)\left(1+2(\widetilde{C}_{\frac{2l}{2-\beta}})^{1-\frac{\beta}{2}}\right)\\
		&\quad\times(2d+16K^2+6K_G^2 +16K^2\widetilde{C}_2+12K_G^2\widetilde{C}_2)^{\frac{\beta}{2}},\\
		\Gamma_5&=(\|h'\|_\infty+\varepsilon)(\|\nabla_yG\|_\infty+\|\nabla_xG\|_\infty)\\
		&\quad\times\sqrt{2d+16K^2+6K_G^2 +16K^2\widetilde{C}_2+12K_G^2\widetilde{C}_2},
	\end{align*}  
	and $\mu_{0}^{\otimes N}$ is the $N$-fold product measure of $\mu_{0}$ with itself. Equivalently, for $\bar d_N:=N^{-1}d_{l^1}$, the same bound reads
	\begin{equation}\label{eq:N-main-bound-normalized}
	\begin{aligned}
	W_{1,\bar d_N}(\mu_{\widetilde X_T^N},\pi^N)
	\leq{}& A_\varepsilon e^{-\widetilde K_\varepsilon T}
	W_{1,\bar d_N}(\mu_0^{\otimes N},\pi^N)\\
	&+\sup_{r>0}\frac{r}{h(r)+\varepsilon r}
	\frac{\Gamma_3\theta^\alpha+\Gamma_4\theta^{\beta/2}+\Gamma_5\theta^{1/2}}
	{\widetilde K_\varepsilon}.
	\end{aligned}
	\end{equation}
	The discretization term in \eqref{eq:N-main-bound-normalized} is uniform over $N\ge2$.

\end{thm}

For later reference, denote the two terms on the right-hand side of \eqref{eq:N-main-bound} by
\begin{align*}
	\mathcal E_{\rm relax}^{N}(T)&:=A_\varepsilon e^{-\widetilde K_\varepsilon T}W_{1,d_{l^1}}(\mu_0^{\otimes N},\pi^N),\\
	\mathcal E_{\rm disc}^{N}(\theta)&:=\sup\limits_{r>0}\frac{r}{h(r)+\varepsilon r}
	\frac{\Gamma_3\theta^\alpha+\Gamma_4\theta^{\frac\beta2}+\Gamma_5\theta^{\frac12}}{\widetilde K_\varepsilon}N.
\end{align*}
\begin{remarks}
The two terms represent the relaxation and discretization errors, respectively; the additional $\theta^{1/2}$-term comes from the temporal discretization of the interaction drift. 

Under the same dimension-scaling assumptions as in Remark~\ref{rem:dim1}, with \(\widetilde C_r\lesssim d^{r/2}\) and dimension-independent model and coupling constants, the discretization bound in Theorem~\ref{mainre2} has a factor of order \(Nd^{\max\left\{l+\beta-\frac12,\frac{\beta+l}{2},\frac12\right\}}\) in the distance \(d_{l^1}\). The additional interaction-discretization term involving \(\Gamma_5\) contributes only a factor \(\sqrt d\) and does not increase this dimension exponent.
\end{remarks}

\begin{remarks}
The finite-time KL bound of \cite[Theorem~1]{MFWB22} concerns ordinary
Euler--Maruyama discretization in the ambient dimension, under a globally
Lipschitz drift with a Lipschitz Jacobian. Its sampling result for the
invariant law additionally assumes a gradient drift and a log-Sobolev
inequality \cite[Corollary~1]{MFWB22}. Formally regarding the entire
$N$-particle system as one diffusion in dimension $Nd$, the published
dimension-dependent bounds do not by themselves provide a discretization
estimate uniform in $N$ for the normalized metric $\bar d_N$; applying the
sampling corollary would also require control of the particle invariant
law's log-Sobolev constant as $N$ varies. Moreover, our polynomially
weighted locally H\"older, possibly superlinear drift need not meet the
global smoothness hypotheses, and the tamed scheme is not the ordinary
Euler scheme covered there. In contrast, \eqref{eq:N-main-bound-normalized}
gives a time-uniform discretization term independent of $N$ under the
present assumptions. This comparison concerns direct applicability of
the published bounds, not what a particle-specific adaptation of their
method might achieve.
\end{remarks}

\begin{thm}\label{thm:IPSmainre2}
	Assume the same assumptions and
	notations as in Theorem \ref{mainre2}. Then, when $0<\theta<\widetilde{\eta},$ $\left\{(\widetilde{X}^{1,N}_{k\theta},\cdots,\widetilde{X}^{N,N}_{k\theta})\right\}_{k\in\mathbb{N}}$ admits a unique invariant distribution denoted by $\pi^{\rm TEM,N}_\theta$ and 
	\begin{bluechange}
	\begin{equation}\label{eq:N-invariant-bound}
		W_{1,d_{l^1}}(\pi_\theta^{\rm TEM,N},\pi^N)\le \mathcal E_{\rm disc}^{N}(\theta).
	\end{equation}
	\end{bluechange}
\end{thm}

\subsection{McKean--Vlasov equation}
When $N \to +\infty$, the law of each particle in the  system \eqref{multiparticle} converges to the law of the following McKean--Vlasov equation:
\begin{equation}\label{mckeanvlasov}
	\begin{cases}
		d\bar{X}_t=\sqrt{2}dB_t+b(\bar{X}_t)dt+ G\circledast\mu_{\bar{X}_t}(\bar{X}_t)dt,\\
		\bar{X}_0\sim\mu_0,
	\end{cases}
\end{equation}
where $G\circledast\mu_{\bar{X}_t}(x)=\int_{\mathbb{R}^d}G(x,y)\mu_{\bar{X}_t}(dy).$
The assumptions of Theorem~\ref{mainre2} imply $m_2<\infty$ and
\[
c_1-\|\nabla_xG\|_\infty-\|\nabla_yG\|_\infty
\ge c_1-\|\nabla_xG\|_\infty-K_G>0.
\]
Consequently, Lemma~\ref{lem:MVmom} yields the time-uniform second-moment bound required below.

\begin{thm}\label{MVTHM}
\begin{bluechange}
	Assume the same assumptions and notations as in Theorem~\ref{mainre2}. Then, for any $N\ge2$, any $\varepsilon>0$ satisfying $1-\varepsilon M-(\|h'\|_\infty+\varepsilon)\|\nabla_yG\|_\infty>0$, any $T>0$, and any $0<\theta<\widetilde\eta$,
	\begin{equation}\label{eq:MV-main-bound}
	\begin{aligned}
		W_1(\mu_{\widetilde X_T^{i,N}},\pi^{MV})
		\leq{}&A_\varepsilon e^{-\widetilde K_\varepsilon T}W_1(\mu_0,\pi^{MV})+\sup\limits_{r>0}\frac{r}{h(r)+\varepsilon r}
		\frac{\Gamma_3\theta^\alpha+\Gamma_4\theta^{\frac\beta2}+\Gamma_5\theta^{\frac12}}{\widetilde K_\varepsilon}\\
		&+\frac{1}{\sqrt{N-1}}\frac{A_\varepsilon}{\widetilde K_\varepsilon}\|\nabla_yG\|_\infty
		\sqrt{\sup\limits_{t\ge0}\mathbb E|\bar X_t|^2}.
	\end{aligned}
	\end{equation}
	The coefficients $A_\varepsilon$, $\widetilde K_\varepsilon$, and $\Gamma_3$--$\Gamma_5$ in \eqref{eq:MV-main-bound} are as specified in Theorem~\ref{mainre2}, while the moment factor is finite by Lemma~\ref{lem:MVmom}.
\end{bluechange}
\end{thm}

\begin{bluechange}
For later reference, denote the three terms on the right-hand side of \eqref{eq:MV-main-bound} by
\begin{align*}
	\mathcal E_{\rm relax}^{\rm MV}(T)&:=A_\varepsilon e^{-\widetilde K_\varepsilon T}W_1(\mu_0,\pi^{MV}),\\
	\mathcal E_{\rm disc}^{\rm MV}(\theta)&:=\sup\limits_{r>0}\frac{r}{h(r)+\varepsilon r}
	\frac{\Gamma_3\theta^\alpha+\Gamma_4\theta^{\frac\beta2}+\Gamma_5\theta^{\frac12}}{\widetilde K_\varepsilon},\\
	\mathcal E_{\rm part}^{\rm MV}(N)&:=\frac{1}{\sqrt{N-1}}\frac{A_\varepsilon}{\widetilde K_\varepsilon}\|\nabla_yG\|_\infty
	\sqrt{\sup\limits_{t\ge0}\mathbb E|\bar X_t|^2}.
\end{align*}
The three terms represent the relaxation, discretization, and finite-particle errors, respectively.
\end{bluechange}
\begin{corollary}\label{coro0}
	Assume the same assumptions and
	notations as in Theorem \ref{MVTHM}. For $N\ge2,$ let $\pi_\theta^{\rm TEM,N,1}$  denote the one-particle marginal distribution of  $\pi_\theta^{\rm TEM,N}.$ Then,
	\begin{bluechange}
	\begin{equation}\label{eq:MV-invariant-marginal-bound}
		W_1(\pi_\theta^{\rm TEM,N,1},\pi^{MV})\le\mathcal E_{\rm disc}^{\rm MV}(\theta)+\mathcal E_{\rm part}^{\rm MV}(N).
	\end{equation}
	\end{bluechange}
\end{corollary}

\begin{corollary}\label{coro1}
Under the assumptions and notations of Theorem~\ref{MVTHM}, let
$T=k\theta$ with $k\in\mathbb N^+$ and let $\mathbb{E}[\mu_{N,k}]$ denotes the expected
	time-averaged empirical measure, i.e., the deterministic measure defined by
	\begin{align*}
		\mathbb{E}[\mu_{N,k}]=\frac{1}{kN}\sum_{i=1}^N\sum_{j=0}^{k-1}\mathbb{E}[\delta_{\widetilde{X}^{i,N}_{j\theta}}]
	\end{align*}. Then
\begin{align}\label{eq:coro1}
	W_1\bigl(\mathbb E[\mu_{N,k}],\pi^{MV}\bigr)
	\leq \mathcal E_{\rm disc}^{\rm MV}(\theta)
	+\frac{C_{\rm poc}}{\sqrt{N-1}}
	+\frac{C_{\rm erg}}{T}
	+\frac{A_\varepsilon W_1(\mu_0,\pi^{MV})\theta}{T},
\end{align}
where
\begin{align*}
    C_{\rm poc}:=&\frac{A_\varepsilon}{\widetilde{K}_\varepsilon}\|\nabla_{y}G\|_\infty \sup\limits_{t\ge0}\sqrt{\mathbb{E}|\bar{X}_t|^2},\\
    C_{\rm erg}:=&\frac{A_\varepsilon}{\widetilde{K}_\varepsilon}W_1(\pi^{MV},\mu_0).
\end{align*}
\end{corollary}

\begin{corollary}\label{coro2}
	Assume the assumptions and notations of Theorem~\ref{MVTHM}.
	\begin{bluechange}
	If $d\in\{1,2\}$, assume additionally that $m_4<\infty$. Then for any $T>0$
	\begin{equation}\label{W1}
		\mathbb E W_1\Big(\frac{1}{N}\sum_{i=1}^N\delta_{\widetilde X_T^{i,N}},\pi^{MV}\Big)
		\leq \mathcal E_{\rm relax}^{\rm MV}(T)+\mathcal E_{\rm disc}^{\rm MV}(\theta)
		+\mathcal E_{\rm part}^{\rm MV}(N)+\mathcal E_{\rm emp}^{\rm MV}(N).
	\end{equation}
	Here,
	\begin{align*}
		\mathcal E_{\rm emp}^{\rm MV}(N):=C_d
		\begin{cases}
			\sup\limits_{t\geq0}(\mathbb E|\bar X_t|^4)^{\frac14}(N^{-\frac12}+N^{-\frac34}), &d=1,\\
			\sup\limits_{t\geq0}(\mathbb E|\bar X_t|^4)^{\frac14}(N^{-\frac12}\ln(1+N)+N^{-\frac34}), &d=2,\\
			\sup\limits_{t\geq0}(\mathbb E|\bar X_t|^2)^{\frac12}(N^{-\frac1d}+N^{-\frac12}), &d\geq3,
		\end{cases}
	\end{align*}
	where $C_d>0$ depends only on the dimension and is independent of $N$, $T$, and $\theta$.
	\end{bluechange}
\end{corollary}

Assume that the initial distribution $\mu_0$ satisfies a Gaussian concentration inequality: there exists a constant $c_G(\mu_0)$ such that for every Lipschitz function $f$,
\begin{equation}\label{eq:concercg}
    \mu_0(e^{f-\mu_0(f)})\le e^{\frac{c_G(\mu_0)}{2}\|f\|^2_{\rm Lip}}.
\end{equation}
Then, the following concentration inequality can be obtained.
\begin{corollary}\label{coro:concen}
    Under the hypotheses of Theorem~\ref{MVTHM}, and assuming the Gaussian concentration inequality \eqref{eq:concercg}, for every $1$-Lipschitz function $f$, every $a>0$, $T=k\theta$ with $k\ge1$ an integer, and $N\ge2$,
    \begin{equation}\label{eq:MV-concentration-bound}
    \begin{aligned}
        &\mathbb{P}\left(\frac{1}{kN}\sum_{i=1}^N\sum_{j=0}^{k-1}f(\widetilde{X}_{j\theta}^{i,N})-\pi^{MV}(f)>a+4\big(\frac{C_{\rm poc}}{\sqrt{N-1}}+\frac{C_{\rm erg}}{T}\big)\right)\\
        \le{}&\frac{4}{a+4\big(\frac{C_{\rm poc}}{\sqrt{N-1}}+\frac{C_{\rm erg}}{T}\big)}(\mathcal E_{\rm disc}^{\rm MV}(\theta)+C_{\rm mom}\sqrt{\theta})\\
		&+\exp\left(-\frac{\big(a+4(\frac{C_{\rm poc}}{\sqrt{N-1}}+\frac{C_{\rm erg}}{T})\big)^2(1-\|\nabla_yG\|_\infty\|h'\|_\infty)^2}{32(h'(0))^2(2+c_G(\mu_0)/T)}NT\right),
    \end{aligned}
    \end{equation}
    where $C_{\rm poc},C_{\rm erg}$ are given in Corollary \ref{coro1} and 
    $$C_{\rm mom}:=(2d+16K^2+6K_G^2 +16K^2\widetilde{C}_2+12K_G^2\widetilde{C}_2)^{\frac12}.$$

\end{corollary}

\begin{example}[Interacting double-well model]\label{ex:interacting-double-well}
\rm The Curie--Weiss double-well model of \cite[Example~2.13]{LWZ21}
provides a genuinely interacting instance of the numerical results above.
To avoid confusion with the H\"older exponent $\beta$ in {\bf (A1)}, write
$a>0$ for the inverse-temperature parameter and $\rho>0$ for the
interaction strength. In dimension $d=1$, let
\[
\begin{aligned}
V(x)&=a\left(\frac{x^4}{4}-\frac{x^2}{2}\right), &
W(x,y)&=-a\rho xy,\\
b(x)&=a(x-x^3), &
G(x,y)&=a\rho y.
\end{aligned}
\]
The interaction is nonzero and depends on the empirical mean of the other
particles. Assumption {\bf (A1)} holds with $\beta=1$, $l=2$, and $L=3a$;
thus $\varkappa=2$ and $K=|b(0)|+2L=6a$. Indeed,
\[
|b(x)-b(y)|\le 3a(1+x^2+y^2)|x-y|.
\]
Since $x^2+xy+y^2\ge (x-y)^2/4$, the two-point drift bound
\eqref{dissipativity} holds with
\[
\beta_0(r)=ar(1-r^2/4),\qquad M=a.
\]
In particular, {\bf (A3)--(A4)} hold. Moreover,
$\|\nabla_xG\|_\infty=0$ and
$\|\nabla_yG\|_\infty=K_G=a\rho$, so {\bf (A5)} holds.
For this profile, \cite[Example~2.13]{LWZ21} gives
$\|h'\|_\infty<e^{a/4}\sqrt{\pi/a}$; hence {\bf (H)} follows whenever
$\rho\sqrt{\pi a}\,e^{a/4}<1$.

It remains to check the extra hypotheses imposed by the tamed scheme.
Fix $\alpha=1/2$, write $s=\sqrt\theta$ and $u=x^2$, and choose any
$c_1>0$ with $s\le a/(2c_1)$. Then
\[
xb_\alpha(x)+c_1x^2
=\frac{-(a-c_1s)u^2+(a+c_1)u}{1+su}
\le \frac{(a+c_1)^2}{2a}.
\]
Thus {\bf (A2)} holds with
$c_2=(a+c_1)^2/(2a)$ and
$\theta^*=(a/(2c_1))^2$.
For example, take
\[
a=\frac1{100},\qquad \rho=\frac1{10},\qquad
c_1=\frac1{25},\qquad c_2=\frac18,\qquad
\theta^*=\frac1{64}.
\]
Now $K=0.06$, $K_G=0.001$, and
\[
K_G\|h'\|_\infty<0.018,
\qquad
c_1-\|\nabla_xG\|_\infty-\sqrt2K_G-8K^2
=0.0112-0.001\sqrt2>0.
\]
Consequently {\bf (H)} and {\bf (A5')} are both satisfied. Substitution
in the definition of $\widetilde\eta$ gives
$\widetilde\eta=\theta^*=1/64$; the admissibility condition for
$\varepsilon$ in Theorem~\ref{MVTHM} holds, for instance, with
$\varepsilon=1$. If $\mu_0=\mathcal N(1,1)$, the required sixth moment
is finite. Theorem~\ref{MVTHM} and Corollary~\ref{coro0} therefore yield,
with constants independent of $N$, $T$, and $\theta$,
\[
\begin{aligned}
W_1(\mu_{\widetilde X_T^{i,N}},\pi^{MV})
&=O\!\left(e^{-\widetilde K_1T}+\theta^{1/2}+(N-1)^{-1/2}\right),
\\
W_1(\pi_\theta^{\rm TEM,N,1},\pi^{MV})
&=O\!\left(\theta^{1/2}+(N-1)^{-1/2}\right)
\end{aligned}
\]
for $N\ge2$ and $0<\theta<1/64$. The continuous-time coupling
conditions for this model are classical; the new conclusion here concerns
the tamed particle discretization and its numerical invariant marginal.
\end{example}

\begin{example}[Local expansion with a H\"older drift]
\label{ex:holder-local-expansion}
\rm In dimension $d=1$, consider
\[
b(x)=x-x^3-\tfrac14\operatorname{sgn}(x)|x|^{1/2},
\qquad G(x,y)=\tfrac1{10}\tanh(x)\tanh(y),
\qquad \operatorname{sgn}(0)=0.
\]
Here $b$ is not locally Lipschitz at zero, whereas
$b'(1/4)+\partial_xG(1/4,z)\ge37/80>0$ for every $z$.
Thus the model exhibits local expansion rather than global dissipativity.
The signed square-root estimate and monotonicity give
{\bf (A1)}, {\bf (A3)--(A5)} with
\[
\beta=\tfrac12,\qquad l=\tfrac52,\qquad
\varkappa=2,\qquad \beta_0(r)=2r-\tfrac14r^3.
\]
For this profile, \eqref{h'} yields $\|h'\|_\infty<e\sqrt\pi$;
hence $\|\nabla_yG\|_\infty\|h'\|_\infty<e\sqrt\pi/10<1/2$,
so {\bf (H)} holds. The negative cubic term ensures {\bf (A2)}
with arbitrarily large $c_1$ upon reducing $\theta^*$, so
{\bf (A5')} can also be met when $\alpha=1/2$.
For a finite initial sixth moment and admissible step sizes,
Corollary~\ref{coro0} therefore gives
\[
W_1(\pi_\theta^{\rm TEM,N,1},\pi^{MV})
=O\!\left(\theta^{1/4}+N^{-1/2}\right).
\]
Unlike the preceding smooth model, this example combines low state
regularity and local expansion with a non-convolution interaction
(since $G(x,x)$ is nonconstant).
It falls outside the hypotheses of the specific tamed Euler convergence
result in \cite[Theorem~1.11]{BH25}, owing to the non-Lipschitz state
drift and the non-convolution interaction. 
\end{example}

\begin{remarks}
	
	\begin{enumerate}
		\item[(1)]{Theorem \ref{MVTHM} gives a finite-time-to-equilibrium error decomposition with three contributions: a discretization term of order $\theta^{\min{\{\alpha,\beta/2\}}}$, a particle approximation term of order $(N-1)^{-\frac{1}{2}}$, and a relaxation term of order $e^{-\widetilde{K}_\varepsilon T}$. Corollary \ref{coro0} then yields the corresponding numerical invariant-measure approximation for the one-particle marginal. In particular, if $\alpha=\frac12$ and $\beta=1$, the invariant marginal bound has the explicit order
			\[
			W_1(\pi_\theta^{\rm TEM,N,1},\pi^{MV})=O(\theta^{1/2}+N^{-1/2}).
			\]
			\blue{For comparison, \cite{MFWB22} obtains sharper KL discretization bounds for Langevin diffusions under global Lipschitz and second-order smoothness assumptions, together with additional moment and Fisher-information conditions. The $W_2$-contraction result of \cite{LMM23}, and its Poincar\'{e}, KL, and total-variation consequences, instead require contractivity at infinity, global Lipschitz regularity, and sufficiently large diffusivity.}}
		
		{Moreover, the constants in the bounds are quantitative and can be traced recursively from the assumptions. The coefficients $A_\varepsilon$, $\widetilde{K}_\varepsilon$, and $\Gamma_i$ are displayed explicitly, while the moment bounds $C_r$ and $\widetilde{C}_r$ are obtained recursively in the appendix from the initial moments and the model parameters.}

		\item[(2)] {Corollary \ref{coro1}
			concerns the deterministic measure obtained after taking expectation of the time-averaged empirical measure. Corollary \ref{coro2} adds the dimension-dependent empirical-measure fluctuation term $\mathcal E_{\rm emp}^{\rm MV}(N)$.}
	\end{enumerate}

\end{remarks}
\section{Proof of the main result for Langevin diffusion}

This section is devoted to proving Theorem \ref{mainre1}. We first establish two a priori estimates for the interpolation process defined in \eqref{interpolation-single}; the proofs are deferred to the appendix for the sake of continuity.
	\begin{lem}\label{pri}
Under the same assumptions and notations as in Theorem~\ref{mainre1}, set
\[
\kappa_\alpha:=
\begin{cases}
c_1,&\alpha\in(0,\frac12),\\
2(c_1-4K^2),&\alpha=\frac12,
\end{cases}
\qquad D_{\rm L}:=2c_2+8K^2.
\]
Then, for every $r\in[2,q]$, the moment constants defined in
\eqref{moment-inter-Langevin} satisfy
\begin{equation}\label{eq:moment-Lan-explicit}
C_r\le
\left[
\max\left\{m_r^{2/r},
\frac{D_{\rm L}+2(d+r-2)}{\kappa_\alpha}\right\}
\right]^{r/2}.
\end{equation}
For $0<r<2$, one has $C_r\le C_2^{r/2}$. In particular,
$C_r<\infty$ for every $r\in(0,q]$.
\end{lem}
\begin{lem}\label{pri2}
	\blue{Under the hypotheses of Theorem~\ref{mainre1}, for any $k\in\mathbb{N}$, $0<\theta<\eta$, and $t\in[k\theta,(k+1)\theta]$, we have}
	\begin{equation*}
		\mathbb{E}[|\widetilde{X}_t-\widetilde{X}_{k\theta}|^2]\leq\theta(8K^2(1+C_2)+2d).
	\end{equation*}
\end{lem}

The preceding two lemmas yield the following distance estimate between the Langevin process and its approximation.
\begin{prop}\label{singlee}
	\blue{Under the hypotheses of Theorem~\ref{mainre1}, for $T>0$, $0<\theta<\eta$, and $\varepsilon>0$ such that $1-\varepsilon M>0$,}
	\begin{equation*}
		W_1(\mu_{\widetilde{X}_T},\mu_{X_{T}})\leq \sup\limits_{r>0}\frac{r}{h(r)+\varepsilon r}\frac{\Gamma_1\theta^{\alpha}+\Gamma_2\theta^{\frac{\beta}{2}}}{K_\varepsilon}.
	\end{equation*}
\end{prop}

\begin{proof}
	For an integer $k\geq0$ and $t\in[k\theta,(k+1)\theta]$, we introduce the following coupling of $X_t$ and $\widetilde{X}_t$.
	
	Given $\delta>0,$ let $\lambda_{\delta},\pi_{\delta}:\mathbb{R}^+\rightarrow[0,1]$ be two Lipschitz continuous functions such that
	\begin{equation}\label{def-para1}
		\lambda_{\delta}(r)^2+\pi_{\delta}(r)^2=1,\forall r\in\mathbb{R}^+,
	\end{equation}
	and
	\begin{equation}\label{def-para2}
		\lambda_{\delta}(r)=\begin{cases}
			1, \text{if } r\geq\delta,\\
			0,\text{if } r\leq\frac{\delta}{2}.
		\end{cases}
	\end{equation}
	\blue{A coupling of $X_t$ and $\widetilde{X}_t$ on $[k\theta,(k+1)\theta]$, with initial values $X_{k\theta}$ and $\widetilde{X}_{k\theta}$, is given by}
	\begin{equation}\label{coupling}
		\begin{split}
			&dX_t=\sqrt{2}[\lambda_{\delta}(r_t)dB^1_t+\pi_{\delta}(r_t)dB_t^2]+b(X_t)dt,\\
			&d\widetilde{X}_t=\sqrt{2}[\lambda_{\delta}(r_t)R_tdB_t^1+\pi_{\delta}(r_t)dB_t^2]+b_{\alpha}(\widetilde{X}_{k\theta})dt,
		\end{split}
	\end{equation}
	where $Z_t:=X_t-\widetilde{X}_t$, $r_t:=|Z_t|$, and $R_t:=I_d-2e_te_t^T$.
	Set $e_t=Z_t/r_t$ if $r_t>0$ and $e_t=e_\star$ if $r_t=0$, where $e_\star$ is a fixed unit vector. The choice at zero does not affect the coupling because $\lambda_\delta=0$ on $[0,\delta/2]$. The cutoff diffusion coefficients are globally Lipschitz, and (A4) gives a one-sided Lipschitz bound for $b$. Localization and the quadratic Lyapunov bound therefore give a nonexplosive strong solution on each grid interval, which can be concatenated from time zero.
	Here $B^1,B^2$ are two independent Brownian motions. In fact, according to L\'evy's characterization of Brownian motions,
	\begin{align*}
		&\hat{B}_t:=\int_{0}^{t}\lambda_{\delta}(|Z_s|)dB_s^{1}+\int_{0}^{t}\pi_{\delta}(|Z_s|)dB_s^{2}\\
		&\check{B}_t:=\int_{0}^{t}\lambda_\delta(|Z_s|)R_sdB_s^{1}+\int_{0}^{t}\pi_{\delta}(|Z_s|)dB_s^{2}
	\end{align*}
		\blue{are two Brownian motions in $\mathbb{R}^d$, so that $(X_t,\widetilde{X}_t)$ is indeed a coupling process.}
	
	By (\ref{coupling}), we have
	\begin{align*}
		dZ_t=(b(X_t)-b_{\alpha}(\widetilde{X}_{k\theta}))dt+2\sqrt{2}\lambda_{\delta}(r_t)e_te_t^TdB_t^1.
	\end{align*}
	Note that $d\tilde{B}_t:=e^T_tdB_t^1$ is a one-dimensional Brownian motion. Then,
	\begin{align*}
		dZ_t=(b(X_t)-b_{\alpha}(\widetilde{X}_{k\theta}))dt+2\sqrt{2}\lambda_{\delta}(r_t)e_td\tilde{B}_t.
	\end{align*}
	\blue{Before applying It\^o's formula to $r_t$, we note that, although $|\cdot|$ is not $C^2$ at $0$, the process $Z_t$ is continuous and has bounded variation near $0$. Therefore, \cite[Lemma~7]{DEGZ20} yields}
	\begin{align*}
		1_{\left\{r_t\neq0\right\}}dr_t&=1_{\left\{r_t\neq0\right\}}\frac{\langle Z_t,dZ_t\rangle}{|Z_t|}+4 \times 1_{\left\{r_t\neq0\right\}}\lambda_{\delta}^2(r_t)\langle(-\frac{Z_tZ_t^T}{r^3_t}+\frac{I_d}{r_t})e_t,e_t\rangle dt.
	\end{align*}
	\blue{A direct calculation gives $\langle(-\frac{Z_tZ_t^T}{r^3_t}+\frac{I_d}{r_t})e_t,e_t\rangle=0$. Therefore, by \eqref{dissipativity-single},}
	\begin{align*}
		1_{\left\{r_t\neq0\right\}}dr_t&=1_{\left\{r_t\neq0\right\}}\langle e_t,b(X_t)-b_{\alpha}(\widetilde{X}_{k\theta})\rangle dt+2\sqrt{2}\lambda_{\delta}(r_t)1_{\left\{r_t\neq0\right\}}d\tilde{B}_t\\
		&=1_{\left\{r_t\neq0\right\}}\langle e_t,b(X_t)-b(\widetilde{X}_{t})\rangle dt
		+1_{\left\{r_t\neq0\right\}}\langle e_t,b(\widetilde{X}_t)-b(\widetilde{X}_{k\theta})\rangle dt\\
		&\quad +1_{\left\{r_t\neq0\right\}}\langle e_t,b(\widetilde{X}_{k\theta})-b_{\alpha}(\widetilde{X}_{k\theta})\rangle dt
		+2\sqrt{2}\lambda_{\delta}(r_t)1_{\left\{r_t\neq0\right\}}d\tilde{B}_t\\
		&\leq \beta_0(r_t)dt+1_{\left\{r_t\neq0\right\}}\langle e_t,b(\widetilde{X}_t)-b(\widetilde{X}_{k\theta})\rangle dt\\
		&\quad +1_{\left\{r_t\neq0\right\}}\langle e_t,b(\widetilde{X}_{k\theta})-b_{\alpha}(\widetilde{X}_{k\theta})\rangle dt
		+2\sqrt{2}\lambda_{\delta}(r_t)1_{\left\{r_t\neq0\right\}}d\tilde{B}_t\\
		&\le\beta_0(r_t)dt+1_{\left\{r_t\neq0\right\}}|b(\widetilde{X}_t)-b(\widetilde{X}_{k\theta})|dt\\
		&\quad +1_{\left\{r_t\neq0\right\}}|b(\widetilde{X}_{k\theta})-b_{\alpha}(\widetilde{X}_{k\theta})|dt
		+2\sqrt{2}\lambda_{\delta}(r_t)1_{\left\{r_t\neq0\right\}}d\tilde{B}_t.
	\end{align*}
	When $r_t=0,$ the coupling degenerates into a synchronous coupling. It follows that
	\begin{align*}
		1_{\left\{r_t=0\right\}}dZ_t=&1_{\left\{r_t=0\right\}}(b(X_t)-b_\alpha(\widetilde{X}_{k\theta}))dt\\
		=&1_{\left\{r_t=0\right\}}(b(\widetilde{X}_t)-b_\alpha(\widetilde{X}_{k\theta}))dt,
	\end{align*}
	\blue{which is a bounded-variation process. Therefore,}
	\begin{align}\label{eq:rt=0}
		1_{\left\{r_t=0\right\}}dr_t\le&1_{\left\{r_t=0\right\}}|b(\widetilde{X}_t)-b_\alpha(\widetilde{X}_{k\theta})|dt\notag\\
		\le&1_{\left\{r_t=0\right\}}(|b(\widetilde{X}_t)-b(\widetilde{X}_{k\theta})|+|b(\widetilde{X}_{k\theta})-b_\alpha(\widetilde{X}_{k\theta})|)dt.
	\end{align}
	Hence,
	\begin{align*}
		dr_t\le&\beta_0(r_t)dt+|b(\widetilde{X}_t)-b(\widetilde{X}_{k\theta})|dt\\
		&\quad +|b(\widetilde{X}_{k\theta})-b_{\alpha}(\widetilde{X}_{k\theta})|dt
		+2\sqrt{2}\lambda_{\delta}(r_t)1_{\left\{r_t\neq0\right\}}d\tilde{B}_t
	\end{align*}
	
	Recalling \eqref{h'} and \eqref{h}, for any $\varepsilon$ such that $1-\varepsilon M>0$, we define 
	\begin{equation*}
		h_\varepsilon(r):=h(r)+\varepsilon r \mbox{ and }H_t^\varepsilon:=e^{K_\varepsilon(t-k\theta)}h_\varepsilon(r_t),
	\end{equation*}
	where $K_\varepsilon$ is a positive constant depending on $\varepsilon$ which will be determined later.
	Then, using It\^o's formula,
	\begin{align}\label{HT}
		dH_t^\varepsilon&=K_\varepsilon H_t^\varepsilon dt+e^{K_\varepsilon (t-k\theta)}(h'(r_t)+\varepsilon)dr_t+4h''(r_t)e^{K_\varepsilon (t-k\theta)}\lambda_\delta^2(r_t)dt\notag\\
		&\leq K_\varepsilon H_t^\varepsilon dt+e^{K_\varepsilon (t-k\theta)}(h'(r_t)+\varepsilon)\beta_0(r_t)dt\notag\\
		&\quad+e^{K_\varepsilon (t-k\theta)}(h'(r_t)+\varepsilon)|b(\widetilde{X}_t)-b(\widetilde{X}_{k\theta})|dt\notag\\
		&\quad+e^{K_\varepsilon (t-k\theta)}(h'(r_t)+\varepsilon)|b(\widetilde{X}_{k\theta})-b_{\alpha}(\widetilde{X}_{k\theta})|dt\notag\\
		&\quad+2\sqrt{2}e^{K_\varepsilon (t-k\theta)}(h'(r_t)+\varepsilon)\lambda_{\delta}(r_t)d\tilde{B}_t+4h''(r_t)e^{K_\varepsilon (t-k\theta)}\lambda_\delta^2(r_t)dt\notag\\
		&=2\sqrt{2}e^{K_\varepsilon (t-k\theta)}(h'(r_t)+\varepsilon)\lambda_{\delta}(r_t)d\tilde{B}_t+K_\varepsilon H_t^\varepsilon dt +e^{K_\varepsilon (t-k\theta)}\mathcal{L}_{\lambda_{\delta}}h(r_t)dt\notag\\
		&\quad +\varepsilon e^{K_\varepsilon (t-k\theta)}\beta_0(r_t)dt+e^{K_\varepsilon (t-k\theta)}(h'(r_t)+\varepsilon) |b(\widetilde{X}_t)-b(\widetilde{X}_{k\theta})| dt\notag \\
		&\quad+e^{K_\varepsilon (t-k\theta)}(h'(r_t)+\varepsilon) |b(\widetilde{X}_{k\theta})-b_{\alpha}(\widetilde{X}_{k\theta})| dt\notag\\
		&:=dM_t+D_t^\varepsilon dt+R_t^\varepsilon dt,
	\end{align}
	where
	\begin{align*}
		&\mathcal{L}_{\lambda_\delta} h(r_t):=4\lambda_\delta^2(r_t) h''(r_t)+\beta_0(r_t) h'(r_t),\\
		&dM_t:=2\sqrt{2}e^{K_\varepsilon (t-k\theta)}(h'(r_t)+\varepsilon)\lambda_{\delta}(r_t)d\tilde{B}_t,\\
		&D_t^\varepsilon:=K_\varepsilon H_t^\varepsilon+e^{K_\varepsilon (t-k\theta)}(\mathcal{L}_{\lambda_{\delta}}h(r_t)+\varepsilon \beta_0(r_t)),\\
		&R_t^\varepsilon:=e^{K_\varepsilon (t-k\theta)}(h'(r_t)+\varepsilon)|b(\widetilde{X}_t)-b(\widetilde{X}_{k\theta})|\\
		&\quad \quad \quad +e^{K_\varepsilon (t-k\theta)}(h'(r_t)+\varepsilon)|b(\widetilde{X}_{k\theta})-b_{\alpha}(\widetilde{X}_{k\theta})|.
	\end{align*}
	Here $M_t$ is a martingale since $|(h'(r)+\varepsilon)\lambda_{\delta}(r)|\le\|h'\|_\infty+\varepsilon.$
	
	First, we deal with $D_t^\varepsilon$. Since by {\bf(A4)} it holds that $\beta_0(r_t)\leq Mr_t,$ we obtain
	\begin{align*}
		D_t^\varepsilon=&K_\varepsilon H_t^\varepsilon+e^{K_\varepsilon (t-k\theta)}[-r_t+(1-\lambda_{\delta}^2(r_t))(r_t+\beta_0(r_t)h'(r_t))\\
		&\qquad+\varepsilon \beta_0(r_t)]\\
		=&e^{K_\varepsilon (t-k\theta)}(1-\lambda_{\delta}^2(r_t))(r_t+\beta_0(r_t)h'(r_t))\\
		&+e^{K_\varepsilon(t-k\theta)}(K_\varepsilon h_\varepsilon(r_t)-r_t+\varepsilon \beta_0(r_t))\\
		\leq&e^{K_\varepsilon (t-k\theta)}(1-\lambda_{\delta}^2(r_t))(r_t+\beta^{+}_0(r_t)h'(r_t))\\
		&+e^{K_\varepsilon (t-k\theta)}[K_\varepsilon(\|h'\|_\infty+\varepsilon)+\varepsilon M-1]r_t.
	\end{align*}
	Taking $K_\varepsilon=\frac{1-\varepsilon M}{\|h'\|_\infty+\varepsilon},$ we have
	\[
	D_t^\varepsilon\leq e^{K_\varepsilon (t-k\theta)}(1-\lambda_{\delta}^2(r_t))(r_t+\beta^+_0(r_t)h'(r_t)).
	\]
	By the basic inequality, Assumption  ({\bf A1}) and Eq.\eqref{prop1},
	\begin{align}\label{Rt1}
		|R_t^\varepsilon|\leq&e^{K_\varepsilon (t-k\theta)}(\|h'\|_\infty+\varepsilon)[L(1+|\widetilde{X}_t|^l+|\widetilde{X}_{k\theta}|^l)|\widetilde{X}_t-\widetilde{X}_{k\theta}|^\beta\notag\\
		&\qquad+2K\theta^\alpha(1+|\widetilde{X}_{k\theta}|^{2\varkappa+1})]\notag\\
		=&e^{K_\varepsilon (t-k\theta)}(\|h'\|_\infty+\varepsilon)[L|\widetilde{X}_t-\widetilde{X}_{k\theta}|^{\beta}+L|\widetilde{X}_t|^l|\widetilde{X}_t-\widetilde{X}_{k\theta}|^\beta\notag\\
		&\qquad+L|\widetilde{X}_{k\theta}|^l|\widetilde{X}_t-\widetilde{X}_{k\theta}|^\beta\notag\\
		&+2K\theta^\alpha(1+|\widetilde{X}_{k\theta}|^{2\varkappa+1})].
	\end{align}
	Since $2l/(2-\beta)\le2\varkappa+1\le q$, Lemma~\ref{pri} supplies the moments required for the local H\"older estimate (with $C_0=1$ if $l=0$). By H\"older's inequality and Lemmas \ref{pri} and \ref{pri2}, we have
	\begin{align}\label{Rt2}
		&\mathbb{E}[|\widetilde{X}_t-\widetilde{X}_{k\theta}|^{\beta}+|\widetilde{X}_t|^l|\widetilde{X}_t-\widetilde{X}_{k\theta}|^\beta+|\widetilde{X}_{k\theta}|^l|\widetilde{X}_t-\widetilde{X}_{k\theta}|^\beta]\notag\\
		\leq&(\mathbb{E}|\widetilde{X}_t-\widetilde{X}_{k\theta}|^2)^{\frac{\beta}{2}}+(\mathbb{E}|\widetilde{X}_t|^{\frac{2l}{2-\beta}})^{1-\frac{\beta}{2}}(\mathbb{E}|\widetilde{X}_t-\widetilde{X}_{k\theta}|^2)^{\frac{\beta}{2}}\notag\\
		&+(\mathbb{E}|\widetilde{X}_{k\theta}|^{\frac{2l}{2-\beta}})^{1-\frac{\beta}{2}}(\mathbb{E}|\widetilde{X}_t-\widetilde{X}_{k\theta}|^2)^{\frac{\beta}{2}}\notag\\
		=&(1+(\mathbb{E}|\widetilde{X}_t|^{\frac{2l}{2-\beta}})^{1-\frac{\beta}{2}}+(\mathbb{E}|\widetilde{X}_{k\theta}|^{\frac{2l}{2-\beta}})^{1-\frac{\beta}{2}})(\mathbb{E}|\widetilde{X}_t-\widetilde{X}_{k\theta}|^2)^{\frac{\beta}{2}}\notag\\
		\leq&(1+2(C_{\frac{2l}{2-\beta}})^{1-\frac{\beta}{2}})[\theta(8K^2(1+C_2)+2d)]^{\frac{\beta}{2}}.
	\end{align}
	The taming contribution is controlled by the other required moment:
	\[
	\mathbb E|\widetilde X_{k\theta}|^{2\varkappa+1}\le C_{2\varkappa+1}<\infty,
	\qquad 2\varkappa+1\le q.
	\]
	Denote
	\begin{align*}
		\Gamma_1&:=2K(\|h'\|_\infty+\varepsilon)(1+C_{2l+2\beta-1}),\\
		\Gamma_2&:=L(\|h'\|_\infty+\varepsilon)(1+2(C_{\frac{2l}{2-\beta}})^{1-\frac{\beta}{2}})[(8K^2(1+C_2)+2d)]^{\frac{\beta}{2}}.
	\end{align*}
	By \eqref{Rt1}, \eqref{Rt2} and  Lemma \ref{pri}, we get
	\begin{equation*}
		\mathbb{E}|R_t^\varepsilon|\leq e^{K_\varepsilon(t-k\theta)}[\Gamma_1\theta^\alpha+\Gamma_2{\theta^{\frac{\beta}{2}}}].
	\end{equation*}
	Then, integrating from $k\theta$ to $(k+1)\theta$ and then taking  expectation, we have
	\begin{align}\label{he}
		\mathbb{E}[e^{K_\varepsilon\theta}h_\varepsilon(r_{(k+1)\theta})]&\leq\mathbb{E}h_\varepsilon(|X_{k\theta}-\widetilde{X}_{k\theta}|)\notag\\
		&\quad +\mathbb{E}\int_{k\theta}^{(k+1)\theta}e^{K_\varepsilon (t-k\theta)}(1-\lambda_{\delta}^2(r_t))(r_t+\beta^+_0(r_t)h'(r_t))dt\notag\\
		&\quad +\mathbb{E}\int_{k\theta}^{(k+1)\theta}|R_t^\varepsilon| dt\notag\\
		&\leq\mathbb{E}h_\varepsilon(|X_{k\theta}-\widetilde{X}_{k\theta}|) +\frac{e^{K_\varepsilon\theta}-1}{K_\varepsilon}[\Gamma_1\theta^{\alpha}+\Gamma_2\theta^{\frac{\beta}{2}}]\notag\\
		&\quad +\mathbb{E}\int_{k\theta}^{(k+1)\theta}e^{K_\varepsilon (t-k\theta)}(1-\lambda_{\delta}^2(r_t))(r_t+\beta^+_0(r_t)h'(r_t))dt.
	\end{align}
	Since $1-\lambda^2_\delta$ is bounded by definition, by Assumptions ({\bf A3}) and ({\bf A4}), it follows that, for any $\delta>0,$ when $r\in(0,\delta),$
	\begin{equation*}
		\beta^+_0(r)h'(r)\leq\|h'\|_{\infty}M^+\delta<\infty.
	\end{equation*}
	Then,
	\begin{align*}
		(1-\lambda_{\delta}^2(r_t))(r_t+\beta^+_0(r_t)h'(r_t))=&(1-\lambda_{\delta}^2(r_t))(r_t+\beta^+_0(r_t)h'(r_t)){1}_{\left\{0<r_t<\delta\right\}}\\
		\le&(1-\lambda_{\delta}^2(r_t))\delta(1+M^+\|h'\|_\infty)\le\delta(1+M^+\|h'\|_\infty).
	\end{align*}
	Therefore, by \eqref{he} we have,
	\begin{align*}
		\mathbb{E}[h_{\varepsilon}(r_{(k+1)\theta})]\leq& e^{-K_\varepsilon\theta}\mathbb{E}[h_\varepsilon(r_{k\theta})] +\mathbb{E}\int_{k\theta}^{(k+1)\theta}(1-\lambda_{\delta}^2(r_t))(r_t+\beta^+_0(r_t)h'(r_t))dt\notag\\
		&\quad +\frac{1-e^{-K_\varepsilon\theta}}{K_\varepsilon}[\Gamma_1\theta^{\alpha}+\Gamma_2\theta^{\frac{\beta}{2}}]\notag\\
		\leq&e^{-K_\varepsilon\theta}\mathbb{E}[h_\varepsilon(r_{k\theta})]+\delta\theta(1+M^+\|h'\|_\infty)+\frac{1-e^{-K_\varepsilon\theta}}{K_\varepsilon}[\Gamma_1\theta^{\alpha}+\Gamma_2\theta^{\frac{\beta}{2}}].
	\end{align*}
	By iterating, we get
	\begin{align*}
		\mathbb{E}[h_{\varepsilon}(r_{(k+1)\theta})]\leq& e^{-(k+1)K_\varepsilon\theta}\mathbb{E}[h_\varepsilon(r_{0})]+\sum_{i=0}^ke^{-(k-i)K_\varepsilon\theta}\theta\delta(1+M^+\|h'\|_\infty)\notag\\
		& +\frac{1-e^{-K_\varepsilon\theta}}{K_\varepsilon}[\Gamma_1\theta^{\alpha}+\Gamma_2\theta^{\frac{\beta}{2}}]\sum_{i=0}^ke^{-iK_\varepsilon\theta}\\
		\leq&e^{-(k+1)K_\varepsilon\theta}\mathbb{E}[h_\varepsilon(r_{0})]+\frac{\delta\theta}{1-e^{-K_\varepsilon\theta}}(1+M^+\|h'\|_\infty)+\frac{\Gamma_1\theta^{\alpha}+\Gamma_2\theta^{\frac{\beta}{2}}}{K_\varepsilon}.
	\end{align*}
	
	Finally, as stated below \eqref{he}, 
	\begin{align*}
		W_1&(\mu_{X_{(k+1)\theta}},\mu_{\widetilde{X}_{(k+1)\theta}})\leq\varliminf\limits_{\delta\rightarrow0}\mathbb{E}[|X_{(k+1)\theta}-\widetilde{X}_{(k+1)\theta}|]\\
		\leq&\sup\limits_{r>0}\frac{r}{h(r)+\varepsilon r}\varliminf\limits_{\delta\rightarrow0}\mathbb{E}[h_\varepsilon(|X_{(k+1)\theta}-\widetilde{X}_{(k+1)\theta}|)]\\
		\leq&\sup\limits_{r>0}\frac{r}{h(r)+\varepsilon r}e^{-(k+1)K_\varepsilon\theta}\mathbb{E}[h_\varepsilon(|X_0-\widetilde{X}_0|)]\\
		&+\sup\limits_{r>0}\frac{r}{h(r)+\varepsilon r}\frac{\Gamma_1\theta^{\alpha}+\Gamma_2\theta^{\frac{\beta}{2}}}{K_\varepsilon}\\
		\leq&\sup\limits_{r>0}\frac{r}{h(r)+\varepsilon r}\sup\limits_{r>0}\frac{h(r)+\varepsilon r}{r}e^{-(k+1)K_\varepsilon\theta}\mathbb{E}[|X_0-\widetilde{X}_0|]\\
		&+\sup\limits_{r>0}\frac{r}{h(r)+\varepsilon r}\frac{\Gamma_1\theta^{\alpha}+\Gamma_2\theta^{\frac{\beta}{2}}}{K_\varepsilon}.
	\end{align*}
	As $X_0=\widetilde{X}_0,$ letting $A_\varepsilon:=\sup\limits_{r>0}\frac{r}{h(r)+\varepsilon r}\sup\limits_{r>0}\frac{h(r)+\varepsilon r}{r}$, we have
	\begin{equation*}
		W_1(\mu_{X_{(k+1)\theta}},\mu_{\widetilde{X}_{(k+1)\theta}})\leq\sup\limits_{r>0}\frac{r}{h(r)+\varepsilon r}\frac{\Gamma_1\theta^{\alpha}+\Gamma_2\theta^{\frac{\beta}{2}}}{K_\varepsilon}.
	\end{equation*}

	More generally, if $T\in[k\theta,(k+1)\theta],$ the same argument on the last subinterval $[k\theta,T]$ yields
	\begin{equation*}
		W_1(\mu_{X_{T}},\mu_{\widetilde{X}_{T}})\leq\sup\limits_{r>0}\frac{r}{h(r)+\varepsilon r}\frac{\Gamma_1\theta^{\alpha}+\Gamma_2\theta^{\frac{\beta}{2}}}{K_\varepsilon},
	\end{equation*}
	because $T-k\theta\leq\theta$ and hence the same one-step bound applies.
\end{proof}
\begin{lem}[\blue{Contraction toward the invariant measure}]\label{singlee2}
	Under the same assumptions and notations as in Theorem~\ref{mainre1}, for any $t\geq0$,
	\[
	W_1(\mu_{X_t},\pi)\leq A_\varepsilon e^{-K_\varepsilon t}W_1(\mu_{X_0},\pi).
	\]
\end{lem}
\begin{proof}
	\blue{We construct a reflection coupling between $X_t$ and $Y_t$:}
	\begin{equation*}
		\begin{split}
			&dY_t=\sqrt{2}R_tdB_t+b(Y_{t})dt,0<t<\tau_c,\ \ Y_t=X_t,t\ge\tau_c,\ Y_0 \sim \mu_{Y_0}
		\end{split}
	\end{equation*}
	where $\tau_c:=\inf\left\{t>0:X_t=Y_t\right\},Z_t:=X_t-Y_t,r_t:=|Z_t|,e_t=\frac{Z_t}{|Z_t|},R_t:=I_d-2e_te_t^T.$ By It\^o's formula,  we have
	\begin{equation*}
		dr_t=1_{\left\{r_t>0\right\}}\langle e_t,b(X_t)-b(Y_{t})\rangle dt+2\sqrt{2}1_{\left\{r_t>0\right\}}d\tilde{B}_t,
	\end{equation*}
	where $d\tilde{B}_t=e_t^TdB_t.$

	Then,	
	\begin{align*}
		{\color{black}d\big(e^{K_\varepsilon t}h_\varepsilon(r_t)\big)}=&1_{\left\{r_t>0\right\}}e^{K_\varepsilon t}[K_\varepsilon h_\varepsilon(r_t)+(h'(r_t)+\varepsilon)\langle e_t,b(X_t)-b(Y_t)\rangle+4h''(r_t)]dt\\
		&+2\sqrt{2}e^{K_\varepsilon t}1_{\left\{r_t>0\right\}}(h'(r_t)+\varepsilon)d\tilde{B}_t.
	\end{align*}
	Omitting the details, which are the same as in the proof of Proposition \ref{singlee}, we obtain
	\begin{align*}
		\mathbb{E}[h_\varepsilon(r_{t})]\leq e^{-K_\varepsilon t}\mathbb{E}[h_\varepsilon(|X_{0}-Y_{0}|)].
	\end{align*}
	According to the definition of $W_1$,
	\begin{align*}
		W_1(\mu_{X_t},\mu_{Y_t})\leq&\mathbb{E}[|X_t-Y_t|]\leq\sup\limits_{r>0}\frac{r}{h(r)+\varepsilon r}\mathbb{E}[h_\varepsilon(|X_t-Y_t|)]\\
		\leq&e^{-K_\varepsilon t}\sup\limits_{r>0}\frac{r}{h(r)+\varepsilon r}\mathbb{E}[h_\varepsilon(|X_0-Y_0|)]\\
		\leq&e^{-K_\varepsilon t}\sup\limits_{r>0}\frac{r}{h(r)+\varepsilon r}\sup\limits_{r>0}\frac{h(r)+\varepsilon r}{r}\mathbb{E}[|X_0-Y_0|]\\
		=&e^{-K_\varepsilon t}A_\varepsilon\mathbb{E}[|X_0-Y_0|].
	\end{align*}
	\blue{Choose an optimal coupling of $X_0$ and $Y_0$ such that $\mu_{Y_0}=\pi$ and $\mathbb{E}|X_0-Y_0|=W_1(\mu_{X_0},\pi)$. Then}
	\begin{equation*}
		W_1(\mu_{X_t},\pi)\leq e^{-K_\varepsilon t}A_\varepsilon W_1(\mu_{X_0},\pi).
	\end{equation*}
\end{proof}

\begin{proof}[\bf Proof of Theorem \ref{mainre1}]
	According to the triangle inequality,
	\begin{equation*}
		W_1(\mu_{\widetilde{X}_T},\pi)\leq W_1(\mu_{\widetilde{X}_T},\mu_{X_T})+W_1(\mu_{X_T},\pi).
	\end{equation*}
	By Proposition \ref{singlee} and  Lemma \ref{singlee2},  Theorem \ref{mainre1} follows.
\end{proof}
\begin{proof}[\bf Proof of Theorem \ref{thm:mainre2}]
	\blue{Let $V(x)=1+|x|^2$ and let $\widetilde{P}_\theta$ denote the one-step transition kernel of the Markov chain $\left\{\widetilde{X}_{k\theta}\right\}_{k\in\mathbb{N}}$. Then, by {\bf (A2)} and \eqref{prop1},}
	\begin{align*}
		\widetilde{P}_\theta V(x)= &\mathbb{E}[V(x+\theta b_\alpha(x)+\sqrt{2\theta}\xi_1)]\\
		=&|x|^2+2\theta\langle x,b_\alpha(x)\rangle+2\theta d+\theta^2|b_\alpha(x)|^2+1\\
		\le&|x|^2-2c_1\theta|x|^2+2c_2\theta+2\theta d+8K^2\theta^{2-2\alpha}(1+|x|^2)+1\\
		=&(1-2c_1\theta+8K^2\theta^{2-2\alpha})|x|^2+2c_2\theta+2\theta d+8K^2\theta^{2-2\alpha}+1.
	\end{align*}
	When $\theta<\eta,$ we have
	\begin{align*}
		\widetilde{P}_\theta V(x)\le&r_1|x|^2+2c_2\theta+2\theta d+8K^2\theta^{2-2\alpha}+1\\
		\le&r_1V(x)+C,
	\end{align*}
	where $r_1$ is given by \eqref{r1} and $C=2c_2\theta+2\theta d+8K^2\theta^{2-2\alpha}+1-r_1$. Taking $\mathscr{C}:=\left\{x:V(x)\le\frac{2C}{1-r_1}\right\}$, we have
	\begin{align*}
		\widetilde{P}_\theta V(x)\le&r_1V(x)+C\\
		=&(r_1-1)V(x)+C+V(x)\\
		=&[(r_1-1)V(x)+C]{1}_{\mathscr{C}}+[(r_1-1)V(x)+C]{1}_{\mathscr{C}^c}+V(x)\\
		\le&C{1}_{\mathscr{C}}+\frac{1}{2}(r_1-1)V(x)+V(x).
	\end{align*}
	Therefore,
	\begin{align*}
		\widetilde{P}_\theta V(x)\le\frac12(1+r_1)V(x)+C{1}_{\mathscr{C}}.
	\end{align*}
	\blue{Since $r_1\in(0,1)$, we have $\frac12(1+r_1)\in(0,1)$. For each fixed $\theta>0$, the nondegenerate Gaussian noise ensures that the one-step transition kernel $\widetilde{P}_{\theta}$ admits a strictly positive, jointly continuous density with respect to Lebesgue measure. Consequently, the Markov chain $\{\widetilde{X}_{k\theta}\}_{k\in\mathbb{N}_0}$ is Lebesgue-irreducible and aperiodic, and every compact set is small. Therefore, Theorem~16.0.1 of Meyn and Tweedie \cite{MT1993} implies that the chain admits a unique invariant probability measure and that there exist constants $\rho\in(0,1)$ and $C>0$ such that, for every $k\in\mathbb{N}$,}
	\begin{align*}
		\sup\limits_{|f|\le V}|\widetilde{P}_\theta^kf(x)-\pi^{\rm TEM}_\theta(f)|\le C\rho^kV(x).
	\end{align*}
	\blue{For a $1$-Lipschitz function $f$ with $f(0)=0$, we have}
	\begin{align*}
		|f(x)|\le |x|\le 1+|x|^2=V(x).
	\end{align*}
	Therefore, by Kantorovich--Rubinstein duality,
	\begin{align*}
		W_1(\mu_{\widetilde{X}_{k\theta}},\pi^{\rm TEM}_\theta)\le C\rho^k\mu_0(V).
	\end{align*}
	Then, from the triangle inequality and Theorem \ref{mainre1}, for any $k\in\mathbb{N},$ we have
	\begin{align*}
		W_1(\pi^{\rm TEM}_\theta,\pi)\le& W_1(\pi^{\rm TEM}_\theta,\mu_{\widetilde{X}_{k\theta}})+W_1(\mu_{\widetilde{X}_{k\theta}},\pi)\\
		\le&C\rho^k\mu_0(V)+\mathcal E_{\rm relax}^{\rm Lan}(k\theta)+\mathcal E_{\rm disc}^{\rm Lan}(\theta).
	\end{align*}
	From Lemma \ref{pri}, $\mu_0(V)<\infty.$ Letting $k\to\infty,$ we complete the proof. 
\end{proof}

\section{Proof of the Main Result for the Interacting Particle System and the McKean--Vlasov Equation}

In this section, we primarily present the proof of {\bf Theorem \ref{mainre2}}. The McKean--Vlasov result is then derived from the particle-system estimate together with known propagation-of-chaos bounds.


\subsection{Proof of the Main Result for the Interacting Particle System}

	\begin{lem}\label{primul}
Under the same assumptions and notations as in Theorem~\ref{mainre2}, set
\[
\widetilde\kappa_\alpha:=
\begin{cases}
c_1-\|\nabla_xG\|_\infty-\sqrt2K_G,
&\alpha\in(0,\frac12),\\
c_1-\|\nabla_xG\|_\infty-\sqrt2K_G-8K^2,
&\alpha=\frac12,
\end{cases}
\]
and
\[
D_{\rm P}:=2c_2+16K^2+6K_G^2+\sqrt2K_G.
\]
Then, for every $r\in[2,q]$, the moment constants defined in
\eqref{moment-inter-Npartical} satisfy
\begin{equation}\label{eq:moment-particle-explicit}
\widetilde C_r\le
\left[
\max\left\{m_r^{2/r},
\frac{D_{\rm P}+2(d+r-2)}{\widetilde\kappa_\alpha}\right\}
\right]^{r/2}.
\end{equation}
For $0<r<2$, one has $\widetilde C_r\le\widetilde C_2^{r/2}$.
In particular, $\widetilde C_r<\infty$ for every $r\in(0,q]$,
uniformly in $N\ge2$ and $0<\theta<\widetilde\eta$.
\end{lem}

\begin{lem}\label{primu2}
	Assume the same assumptions and notations as in Theorem~\ref{mainre2}. Then, for any $k\in\mathbb N$ and $t\in[k\theta,(k+1)\theta]$,
	\begin{equation*}
		\sum_{i=1}^N\mathbb{E}[|\widetilde{X}^{i,N}_t-\widetilde{X}^{i,N}_{k\theta}|^2]\leq N\theta(2d+16K^2+6K_G^2 +16K^2\widetilde{C}_2+12K_G^2\widetilde{C}_2).
	\end{equation*}
\end{lem}

\blue{The preceding two lemmas yield the following distance estimate between the discrete and continuous processes.}
\begin{prop}\label{multi}
	Assume the same assumptions and notations as in Theorem~\ref{mainre2}. Then, for any $\varepsilon>0$ such that $1-\varepsilon M-(\|h'\|_\infty+\varepsilon)\|\nabla_yG\|_\infty>0$ and any $T>0$,
	we have
	\begin{align*}
		W_{1,d_{l^1}}(\mu_{X^N_{T}},\mu_{\widetilde{X}^N_{T}})\leq\sup\limits_{r>0}\frac{r}{h(r)+\varepsilon r}\frac{N(\Gamma_3\theta^{\alpha}+\Gamma_4\theta^{\frac{\beta}{2}}+\Gamma_5\theta^{\frac{1}{2}})}{\widetilde{K}_\varepsilon},
	\end{align*}
	\blue{where $\mu_{X^N_T}$ denotes the law of $X^N_T=(X_T^{1,N},\ldots,X_T^{N,N})$, the solution of \eqref{multiparticle}.}
\end{prop}

\begin{proof}
	Fix an integer $k\ge0$, $t\in[k\theta,(k+1)\theta]$, and $\delta>0$. We construct a coupling of $X_t^{i,N}$ and $\widetilde X_t^{i,N}$ similarly to that in the proof of Proposition~\ref{singlee}. The functions $\lambda_\delta$ and $\pi_\delta$ are defined by \eqref{def-para1} and \eqref{def-para2}.
	
	A coupling of the solutions to \eqref{multiparticle} and \eqref{interpolation}, with initial values $X^{i,N}_{k\theta},\widetilde{X}^{i,N}_{k\theta}\in\mathbb{R}^d$ for $1\leq i\leq N$, is given by a strong solution of the following system:
	\begin{equation}\label{copling-par}
		\begin{split}
			dX_t^{i,N}=&\sqrt{2}[\lambda_{\delta}(|Z_t^i|)dB_t^{1,i}+\pi_{\delta}(|Z_t^i|)dB_t^{2,i}]+b(X_t^{i,N})dt\\
			&+\frac{1}{N-1}\sum\limits_{j\neq i}G(X_{t}^{i,N},X_{t}^{j,N})dt\\
			d\widetilde{X}_{t}^{i,N}=&\sqrt{2}[\lambda_{\delta}(|Z_t^i|)R_t^idB_t^{1,i}+\pi_{\delta}(|Z_t^i|)dB_t^{2,i}]+b_{\alpha}(\widetilde{X}_{k\theta}^{i,N})dt\\
			&+\frac{1}{N-1}\sum\limits_{j\neq i}G(\widetilde{X}_{k\theta}^{i,N},\widetilde{X}_{k\theta}^{j,N})dt.
		\end{split}
	\end{equation}
	\blue{Here $Z_t^i:=X_t^{i,N}-\widetilde{X}_{t}^{i,N}$ and $R_t^i:=I_d-2e^i_t(e_t^i)^T$, where $I_d$ is the $d$-dimensional identity matrix and $e^i_t(e^i_t)^T$ is the orthogonal projection onto the unit vector $e_t^i:=Z_t^i/|Z_t^i|$ when $|Z_t^i|\neq0$. At $Z_t^i=0$, set $e_t^i=e_\star$, using the same fixed unit vector for every particle. As in the single-particle construction, the cutoff makes the diffusion coefficients Lipschitz; (A4)--(A5) give the required one-sided drift bound. Concatenation from $k=0$ produces an exchangeable coupling on $[0,\infty)$.} Moreover, $B_t^1=(B_t^{1,i})_{1\leq i\leq N}$ and $B_t^2=(B_t^{2,i})_{1\leq i\leq N}$ are two independent standard Brownian motions in $(\mathbb{R}^d)^N$.
	
	By L\'evy's characterization,  for $1\leq i\leq N,$
	\begin{equation*}
		\begin{split}
			&\hat{B}_t^i:=\int_{0}^{t}\lambda_{\delta}(|Z_s^i|)dB_s^{1,i}+\int_{0}^{t}\pi_{\delta}(|Z_s^i|)dB_s^{2,i},\\
			&\check{B}_t^i:=\int_{0}^{t}\lambda_\delta(|Z_s^i|)R_s^idB_s^{1,i}+\int_{0}^{t}\pi_{\delta}(|Z_s^i|)dB_s^{2,i},
		\end{split}
	\end{equation*}
		\blue{are two standard Brownian motions in $\mathbb{R}^d$, so that $(X_t^{i,N},\widetilde{X}_{t}^{i,N})$ is indeed a coupling process.}
	
	Then by \eqref{copling-par},
	\begin{align*}
		dZ_t^i=&2\sqrt{2}\lambda_{\delta}(|Z_t^i|)e^i_td\tilde{B}^i_t+[b(X_t^{i,N})-b_{\alpha}(\widetilde{X}_{k\theta}^{i,N})]dt\\
		&+\frac{1}{N-1}\sum\limits_{j\neq i}[G(X_t^{i,N},X_t^{j,N})-G(\widetilde{X}_{k\theta}^{i,N},\widetilde{X}_{k\theta}^{j,N})]dt,
	\end{align*}
	\blue{where $d\tilde{B}^i_t=(e_t^i)^T dB_t^{1,i}$ and $\tilde B^i$, $1\leq i\leq N$, are one-dimensional standard Brownian motions.}

	Let $r_t^i=|Z_t^i|,1\leq i\leq N.$ By It\^o's formula,
	\begin{align*}
		1_{\left\{ r_t^i\neq0\right\}}dr_t^i=&1_{\left\{ r_t^i\neq0\right\}}2\sqrt{2}\lambda_\delta(r_t^i)d\tilde{B}_t^i \\
		&+1_{\left\{ r_t^i\neq0\right\}}\langle e_t^i,b(X_t^{i,N})-b_{\alpha}(\widetilde{X}_{k\theta}^{i,N})\rangle dt\\
		&+1_{\left\{ r_t^i\neq0\right\}}\langle e_t^i,\frac{1}{N-1}\sum\limits_{j\neq i}[G(X_t^{i,N},X_t^{j,N})\\
		&\qquad-G(\widetilde{X}_{k\theta}^{i,N},\widetilde{X}_{k\theta}^{j,N})]\rangle dt\\
		&+1_{\left\{ r_t^i\neq0\right\}}\sum_{k,l=1}^{d}[1_{\left\{k=l\right\}}(r_t^i)^{-1}\\
		&\qquad-(X_t^{i,N,k}-\widetilde{X}^{i,N,k}_{t})(X_t^{i,N,l}-\widetilde{X}_{t}^{i,N,l})(r_t^i)^{-3}]\notag\\
		&\qquad \times\lambda^2_{\delta}(r_t^i)(I_d-R_t^i)^2_{kl}dt,
	\end{align*}
	where $X_t^{i,N,k}$ and $\widetilde{X}^{i,N,k}$ denote the $k$-th coordinates of $X_t^{i,N}$ and $\widetilde{X}_t^{i,N}$, respectively. A direct calculation gives
	\begin{align*}
		&1_{\left\{ r_t^i\neq0\right\}}\sum_{k,l=1}^{d}[1_{\left\{k=l\right\}}(r_t^i)^{-1}\notag\\
		&\quad -(X_t^{i,N,k}-\widetilde{X}^{i,N,k}_{t})(X_t^{i,N,l}-\widetilde{X}_{t}^{i,N,l})(r_t^i)^{-3}]\\
		&\qquad\lambda^2_{\delta}(r_t^i)(I_d-R_t^i)^2_{kl}=0.
	\end{align*}
	In order to use \eqref{dissipativity}, we rewrite $dr_t^i$ as
	\begin{align}\label{esti-r1}
		1_{\left\{ r_t^i\neq0\right\}}dr_t^i=&1_{\left\{ r_t^i\neq0\right\}}2\sqrt{2}\lambda_\delta(r_t^i)d\tilde{B}_t^i\notag \\
		&+1_{\left\{ r_t^i\neq0\right\}}\langle e_t^i,b(X_t^{i,N})-b(\widetilde{X}_{t}^{i,N})\rangle dt\notag\\
		&+1_{\left\{ r_t^i\neq0\right\}}\langle e_t^i,b(\widetilde{X}_{t}^{i,N})-b_{\alpha}(\widetilde{X}_{k\theta}^{i,N})\rangle dt\notag\\
		&+1_{\left\{ r_t^i\neq0\right\}}\langle e_t^i,\frac{1}{N-1}\sum\limits_{j\neq i}[G(X_t^{i,N},X_t^{j,N})-G(X_t^{i,N},\widetilde{X}_t^{j,N})\notag\\
		&+G(X_t^{i,N},\widetilde{X}_t^{j,N})-G(\widetilde{X}_t^{i,N},\widetilde{X}_t^{j,N})\notag\\
		&+G(\widetilde{X}_t^{i,N},\widetilde{X}_t^{j,N})-G(\widetilde{X}_{t}^{i,N},\widetilde{X}_{k\theta}^{j,N})\notag\\
		&+G(\widetilde{X}_{t}^{i,N},\widetilde{X}_{k\theta}^{j,N})-G(\widetilde{X}_{k\theta}^{i,N},\widetilde{X}_{k\theta}^{j,N})]\rangle dt.
	\end{align}
	Then, by \eqref{dissipativity},
	\begin{align}\label{esti-r2}
		&1_{\left\{ r_t^i\neq0\right\}}\langle e_t^i,b(X_t^{i,N})-b(\widetilde{X}_{t}^{i,N})\notag\\
		&\qquad+\frac{1}{N-1}\sum\limits_{j\neq i}[G(X_t^{i,N},\widetilde{X}_t^{j,N})-G(\widetilde{X}_t^{i,N},\widetilde{X}_t^{j,N})]\rangle dt\notag\\
		&\leq 1_{\left\{ r_t^i\neq0\right\}}\beta_0(r_t^i)dt.
	\end{align}
	Applying the mean value theorem, it follows that
	\begin{align}\label{esti-r3}
		&1_{\left\{ r_t^i\neq0\right\}}\langle e_t^i,\frac{1}{N-1}\sum\limits_{j\neq i}[G(X_t^{i,N},X_t^{j,N})-G(X_t^{i,N},\widetilde{X}_t^{j,N})]\rangle dt \notag\\
		\leq&1_{\left\{ r_t^i\neq0\right\}}\frac{1}{N-1}\|\nabla_{y}G\|_{\infty}\sum\limits_{j\neq i}r_t^jdt;
	\end{align}
	\begin{align}\label{esti-r4}
		&1_{\left\{ r_t^i\neq0\right\}}\langle e_t^i,\frac{1}{N-1}\sum\limits_{j\neq i}[G(\widetilde{X}_t^{i,N},\widetilde{X}_t^{j,N})-G(\widetilde{X}_{t}^{i,N},\widetilde{X}_{k\theta}^{j,N})]\rangle dt\notag\\
		\leq&1_{\left\{ r_t^i\neq0\right\}}\frac{1}{N-1}\|\nabla_{y}G\|_{\infty}\sum\limits_{j\neq i}|\widetilde{X}_t^{j,N}-\widetilde{X}_{k\theta}^{j,N}|dt;
	\end{align}
	\begin{align}\label{esti-r5}
		&1_{\left\{ r_t^i\neq0\right\}}\langle e_t^i,\frac{1}{N-1}\sum\limits_{j\neq i}[G(\widetilde{X}_{t}^{i,N},\widetilde{X}_{k\theta}^{j,N})-G(\widetilde{X}_{k\theta}^{i,N},\widetilde{X}_{k\theta}^{j,N})]\rangle dt\notag\\
		\leq&1_{\left\{ r_t^i\neq0\right\}}\frac{1}{N-1}\|\nabla_{x}G\|_{\infty}\sum\limits_{j\neq i}|\widetilde{X}_t^{i,N}-\widetilde{X}_{k\theta}^{i,N}|dt.
	\end{align}
	Combining \eqref{esti-r1}-\eqref{esti-r5}, we get
	\begin{align}\label{rt1}
		1_{\left\{ r_t^i\neq0\right\}}dr_t^i\leq&1_{\left\{ r_t^i\neq0\right\}}2\sqrt{2}\lambda_\delta(r_t^i)d\tilde{B}_t^i\nonumber\\
		&+1_{\left\{ r_t^i\neq0\right\}}\beta_0(r_t^i)dt\nonumber\\
		&+1_{\left\{ r_t^i\neq0\right\}}\langle e_t^i,b(\widetilde{X}_{t}^{i,N})-b_{\alpha}(\widetilde{X}_{k\theta}^{i,N})\rangle dt\nonumber\\
		&+1_{\left\{ r_t^i\neq0\right\}}\frac{1}{N-1}\|\nabla_{y}G\|_{\infty}\sum\limits_{j\neq i}r_t^jdt\nonumber\\
		&+1_{\left\{ r_t^i\neq0\right\}}\frac{1}{N-1}\|\nabla_{y}G\|_{\infty}\sum\limits_{j\neq i}|\widetilde{X}_t^{j,N}-\widetilde{X}_{k\theta}^{j,N}|dt\nonumber\\
		&+1_{\left\{ r_t^i\neq0\right\}}\frac{1}{N-1}\|\nabla_{x}G\|_{\infty}\sum\limits_{j\neq i}|\widetilde{X}_t^{i,N}-\widetilde{X}_{k\theta}^{i,N}|dt.
	\end{align}
	\blue{For $r_t^i=0$, an argument analogous to that used in \eqref{eq:rt=0} gives}
	\begin{align}\label{rt}
		dr_t^i\leq&1_{\left\{ r_t^i\neq0\right\}}2\sqrt{2}\lambda_\delta(r_t^i)d\tilde{B}_t^i+\beta_0(r_t^i)dt\nonumber\\
		&+|b(\widetilde{X}_{t}^{i,N})-b_{\alpha}(\widetilde{X}_{k\theta}^{i,N})| dt\nonumber\\
		&+\frac{1}{N-1}\|\nabla_{y}G\|_{\infty}\sum\limits_{j\neq i}r_t^jdt+\frac{1}{N-1}\|\nabla_{y}G\|_{\infty}\sum\limits_{j\neq i}|\widetilde{X}_t^{j,N}-\widetilde{X}_{k\theta}^{j,N}|dt\nonumber\\
		&+\frac{1}{N-1}\|\nabla_{x}G\|_{\infty}\sum\limits_{j\neq i}|\widetilde{X}_t^{i,N}-\widetilde{X}_{k\theta}^{i,N}|dt.
	\end{align}

	\blue{Recall that, for any $\varepsilon>0$, $h_\varepsilon(r):=h(r)+\varepsilon r$ for all $r\geq0$. For $t\in[k\theta,(k+1)\theta]$, let}
	\begin{equation*}
		\widetilde{H}_t^\varepsilon:=e^{\widetilde{K}_\varepsilon (t-k\theta)}\sum_{i=1}^{N}h_\varepsilon(r_t^i),
	\end{equation*}
	\blue{where $\widetilde{K}_\varepsilon$ is a constant depending on $\varepsilon$, to be determined below. Applying It\^o's formula and using \eqref{rt}, we obtain}
	\begin{align}\label{H-est-1}
		d\widetilde{H}_t^\varepsilon=&\widetilde{K}_\varepsilon \widetilde{H}_t^\varepsilon dt+e^{\widetilde{K}_\varepsilon (t-k\theta)}\sum_{i=1}^{N}(h'(r_t^i)+\varepsilon)dr_t^i\notag\\
		&+4e^{\widetilde{K}_\varepsilon (t-k\theta)}\sum_{i=1}^Nh''(r_t^i)1_{\left\{ r_t^i\neq0\right\}}\lambda^2_\delta(r_t^i)dt\notag\\
		\leq&\widetilde{K}_\varepsilon \widetilde{H}_t^\varepsilon dt\notag\\
		&+e^{\widetilde{K}_\varepsilon (t-k\theta)}\sum_{i=1}^N(h'(r_t^i)+\varepsilon)2\sqrt{2}\lambda_{\delta}(r_t^i)d\tilde{B}_t^i+e^{\widetilde{K}_\varepsilon (t-k\theta)}\sum_{i=1}^N(h'(r_t^i)+\varepsilon)\beta_0(r_t^i)dt\notag\\
		&+e^{\widetilde{K}_\varepsilon (t-k\theta)}\sum_{i=1}^N(h'(r_t^i)+\varepsilon)\frac{1}{N-1}\|\nabla_{y}G\|_{\infty}\sum\limits_{j\neq i}r_t^jdt\notag\\
		&+e^{\widetilde{K}_\varepsilon (t-k\theta)}\sum_{i=1}^N(h'(r_t^i)+\varepsilon)|b(\widetilde{X}_{t}^{i,N})-b_\alpha(\widetilde{X}_{k\theta}^{i,N})| dt\notag\\
		&+e^{\widetilde{K}_\varepsilon (t-k\theta)}\sum_{i=1}^N(h'(r_t^i)+\varepsilon)\frac{1}{N-1}\|\nabla_{y}G\|_{\infty}\sum\limits_{j\neq i}|\widetilde{X}_t^{j,N}-\widetilde{X}_{k\theta}^{j,N}|dt\notag\\
		&+e^{\widetilde{K}_\varepsilon (t-k\theta)}\sum_{i=1}^N(h'(r_t^i)+\varepsilon)\frac{1}{N-1}\|\nabla_{x}G\|_{\infty}\sum\limits_{j\neq i}|\widetilde{X}_t^{i,N}-\widetilde{X}_{k\theta}^{i,N}|dt\notag\\
		&+4e^{\widetilde{K}_\varepsilon (t-k\theta)}\sum_{i=1}^Nh''(r_t^i)\lambda_{\delta}^2(r_t^i)dt\notag\\
		\le&\widetilde{K}_\varepsilon \widetilde{H}_t^\varepsilon dt+e^{\widetilde{K}_\varepsilon (t-k\theta)}\sum_{i=1}^N(h'(r_t^i)+\varepsilon)2\sqrt{2}\lambda_{\delta}(r_t^i)d\tilde{B}_t^i\notag\\
		&+e^{\widetilde{K}_\varepsilon (t-k\theta)}\sum_{i=1}^N(\mathcal{L}_{\lambda_{\delta}}h(r_t^i)+\varepsilon \beta_0(r_t^i))dt\notag\\
		&+e^{\widetilde{K}_\varepsilon (t-k\theta)}\sum_{i=1}^N(\|h'\|_\infty+\varepsilon)\frac{1}{N-1}\|\nabla_{y}G\|_{\infty}\sum\limits_{j\neq i}r_t^jdt\notag\\
		&+e^{\widetilde{K}_\varepsilon (t-k\theta)}\sum_{i=1}^N(h'(r_t^i)+\varepsilon)|b(\widetilde{X}_{t}^{i,N})-b_{\alpha}(\widetilde{X}_{k\theta}^{i,N})| dt\notag\\
		&+e^{\widetilde{K}_\varepsilon (t-k\theta)}\sum_{i=1}^N(\|h'\|_\infty+\varepsilon)\frac{1}{N-1}(\|\nabla_{y}G\|_{\infty}+\|\nabla_{x}G\|_{\infty})\notag\\
		&\qquad\times\sum\limits_{j\neq i}|\widetilde{X}^{j,N}_t-\widetilde{X}_{k\theta}^{j,N}|dt\notag\\
		:=&d\widetilde{M}^{\varepsilon}_t+\widetilde{D}^{\varepsilon}_tdt+\widetilde{R}^{\varepsilon}_tdt,
	\end{align}
	where
	\begin{align*}
		d\widetilde{M}_t^\varepsilon&=2\sqrt{2}e^{\widetilde{K}_\varepsilon (t-k\theta)}\sum_{i=1}^N(h'(r_t^i)+\varepsilon)\lambda_{\delta}(r_t^i)d\tilde{B}_t^i,
	\end{align*}
	\begin{align*}
		\widetilde{D}_t^\varepsilon&:=\widetilde{K}_\varepsilon \widetilde{H}_t^\varepsilon\\
		&\quad+e^{\widetilde{K}_\varepsilon (t-k\theta)}\sum_{i=1}^N\Big(\mathcal{L}_{\lambda_{\delta}}h(r_t^i)+\varepsilon \beta_0(r_t^i) \\
		&\qquad+\frac{\|h'\|_\infty+\varepsilon}{N-1}\|\nabla_{y}G\|_{\infty}\sum\limits_{j\neq i}r_t^j\Big),
	\end{align*}
	and
	\begin{align*}
		\widetilde{R}_t^\varepsilon:=&e^{\widetilde{K}_\varepsilon (t-k\theta)}\sum_{i=1}^N\Bigg[(h'(r_t^i)+\varepsilon)|b(\widetilde{X}_{t}^{i,N})-b(\widetilde{X}_{k\theta}^{i,N})\\
		&\qquad+b(\widetilde{X}_{k\theta}^{i,N})-b_{\alpha}(\widetilde{X}_{k\theta}^{i,N})|\\
		&\quad +(\|h'\|_\infty+\varepsilon)\sum\limits_{j\neq i}\frac{\|\nabla_{y}G\|_{\infty}|\widetilde{X}^{j,N}_t-\widetilde{X}_{k\theta}^{j,N}|+\|\nabla_{x}G\|_{\infty}|\widetilde{X}^{i,N}_t-\widetilde{X}_{k\theta}^{i,N}|}{N-1}\Bigg].
	\end{align*}
	Here $\widetilde{M}_t^\varepsilon$ is a martingale.
	Then, by Assumption ({\bf A4}), 
	\begin{align*}
		\widetilde{D}_t^\varepsilon
		\leq&\widetilde{K}_\varepsilon \widetilde{H}_t^\varepsilon\\
		&+e^{\widetilde{K}_\varepsilon (t-k\theta)}\sum_{i=1}^N(\mathcal{L}_{\lambda_{\delta}}h(r_t^i)+\varepsilon \beta_0(r_t^i))+e^{\widetilde{K}_\varepsilon (t-k\theta)}({\|h'\|_{\infty}+\varepsilon})\|\nabla_{y}G\|_{\infty}\sum_{i=1}^N r_t^i\\
		\leq&e^{\widetilde{K}_\varepsilon (t-k\theta)}\sum_{i=1}^N[1-\lambda_\delta^2(r_t^i)][r_t^i+\beta_0(r_t^i)h'(r_t^i)]\\
		&+e^{\widetilde{K}_\varepsilon (t-k\theta)}\sum_{i=1}^N[\widetilde{K}_\varepsilon(h(r_t^i)+\varepsilon r_t^i)+\varepsilon \beta_0(r_t^i)\\
		&\qquad+(\|h'\|_\infty+\varepsilon)\|\nabla_{y}G\|_{\infty}r_t^i-r_t^i]\\
		\leq&e^{\widetilde{K}_\varepsilon (t-k\theta)}\sum_{i=1}^N[1-\lambda_\delta^2(r_t^i)][r_t^i+\beta_0(r_t^i)h'(r_t^i)]\\
		&+e^{\widetilde{K}_\varepsilon (t-k\theta)}\sum_{i=1}^N[\widetilde{K}_\varepsilon(\|h'\|_\infty+\varepsilon)+\varepsilon M\\
		&\qquad+(\|h'\|_\infty+\varepsilon)\|\nabla_{y}G\|_{\infty}-1]r_t^i.
	\end{align*}
	Let
	\[
	\widetilde{K}_\varepsilon=\frac{1-\varepsilon M-(\|h'\|_\infty+\varepsilon)\|\nabla_{y}G\|_{\infty}}{\|h'\|_\infty+\varepsilon}.
	\]
	\blue{Then}
	\begin{equation}\label{H-est-D}
		\widetilde{D}_t^\varepsilon\leq e^{\widetilde{K}_\varepsilon (t-k\theta)}\sum_{i=1}^N[1-\lambda_\delta^2(r_t^i)][r_t^i+\beta^+_0(r_t^i)h'(r_t^i)].
	\end{equation}
	For $\widetilde{R}_t^\varepsilon$, by Assumption ({\bf A1}) and the last inequality in \eqref{prop1}, we have
	\begin{align*}
		|\widetilde{R}_t^\varepsilon|\leq&e^{\widetilde{K}_\varepsilon (t-k\theta)}(\|h'\|_\infty+\varepsilon)\sum_{i=1}^N\Big[L(1+|\widetilde{X}^{i,N}_t|^{l}+|\widetilde{X}_{k\theta}^{i,N}|^{l})|\widetilde{X}^{i,N}_t-\widetilde{X}_{k\theta}^{i,N}|^{\beta}\\
		&\quad +2K\theta^\alpha(1+|\widetilde{X}^{i,N}_{k\theta}|^{2\varkappa+1})\\
		&\quad +\sum\limits_{j\neq i}\frac{\|\nabla_{y}G\|_{\infty}|\widetilde{X}^{j,N}_t-\widetilde{X}_{k\theta}^{j,N}|+\|\nabla_{x}G\|_{\infty}|\widetilde{X}^{i,N}_t-\widetilde{X}_{k\theta}^{i,N}|}{N-1}\Big]\\
		=&e^{\widetilde{K}_\varepsilon (t-k\theta)}(\|h'\|_\infty+\varepsilon)\Bigg(\sum_{i=1}^NL\Big [|\widetilde{X}^{i,N}_t-\widetilde{X}_{k\theta}^{i,N}|^{\beta}\\
		&\quad +|\widetilde{X}_t^{i,N}|^{l}|\widetilde{X}^{i,N}_t-\widetilde{X}_{k\theta}^{i,N}|^{\beta}+|\widetilde{X}_{k\theta}^{i,N}|^l|\widetilde{X}^{i,N}_t-\widetilde{X}_{k\theta}^{i,N}|^\beta\Big]\\
		&\quad +2K\theta^\alpha N+2K\theta^\alpha\sum_{i=1}^N|\widetilde{X}^{i,N}_{k\theta}|^{2\varkappa+1}\Bigg)\\
		&\quad +e^{\widetilde{K}_\varepsilon (t-k\theta)}(\|h'\|_\infty+\varepsilon)(\|\nabla_{y}G\|_{\infty}+\|\nabla_{x}G\|_{\infty})\sum_{i=1}^N|\widetilde{X}^{i,N}_t-\widetilde{X}_{k\theta}^{i,N}|.
	\end{align*}
	The inequality $2l/(2-\beta)\le2\varkappa+1\le q$ and Lemma~\ref{primul} justify all moment factors in the following H\"older estimate, using $\widetilde C_0=1$ if $l=0$.
	\blue{Applying H\"older's inequality and Lemmas~\ref{primul} and \ref{primu2}, we obtain}
	\begin{align*}
		&\mathbb{E}\sum_{i=1}^N\Big[|\widetilde{X}^{i,N}_t-\widetilde{X}_{k\theta}^{i,N}|^{\beta}+|\widetilde{X}_t^{i,N}|^{l}|\widetilde{X}^{i,N}_t-\widetilde{X}_{k\theta}^{i,N}|^{\beta}+|\widetilde{X}_{k\theta}^{i,N}|^l|\widetilde{X}_{t}^{i,N}-\widetilde{X}_{k\theta}^{i,N}|^\beta\Big]\\
		\leq&\Big[N^{1-\frac{\beta}{2}}+(\mathbb{E}\sum_{i=1}^N|\widetilde{X}_t^{i,N}|^\frac{2l}{2-\beta})^{1-\frac{\beta}{2}}+(\mathbb{E}\sum_{i=1}^N|\widetilde{X}_{k\theta}^{i,N}|^\frac{2l}{2-\beta})^{1-\frac{\beta}{2}}\Big]\\
		&\qquad\times\Big(\mathbb{E}\sum_{i=1}^N|\widetilde{X}^{i,N}_t-\widetilde{X}_{k\theta}^{i,N}|^2\Big)^\frac{\beta}{2}\\
		=&\Big[N^{1-\frac{\beta}{2}}+(N\mathbb{E}|\widetilde{X}_t^{1,N}|^\frac{2l}{2-\beta})^{1-\frac{\beta}{2}} +(N\mathbb{E}|\widetilde{X}_{k\theta}^{1,N}|^\frac{2l}{2-\beta})^{1-\frac{\beta}{2}}\Big]\\
		&\qquad\times\Big(\mathbb{E}\sum_{i=1}^N|\widetilde{X}^{i,N}_t-\widetilde{X}_{k\theta}^{i,N}|^2\Big)^\frac{\beta}{2}\\
		=&N^{1-\frac{\beta}{2}}\Big[1+(\mathbb{E}|\widetilde{X}_t^{1,N}|^\frac{2l}{2-\beta})^{1-\frac{\beta}{2}}+(\mathbb{E}|\widetilde{X}_{k\theta}^{1,N}|^\frac{2l}{2-\beta})^{1-\frac{\beta}{2}}\Big]\\
		&\qquad\times\Big(\mathbb{E}\sum_{i=1}^N|\widetilde{X}^{i,N}_t-\widetilde{X}_{k\theta}^{i,N}|^2\Big)^\frac{\beta}{2}\\
		\leq&N^{1-\frac{\beta}{2}}(1+2(\widetilde{C}_{\frac{2l}{2-\beta}})^{1-\frac{\beta}{2}})\\
		&\qquad\times[N\theta(2d+16K^2+6K_G^2 +16K^2\widetilde{C}_2+12K_G^2\widetilde{C}_2)]^\frac{\beta}{2}.
	\end{align*}
	\blue{By the Cauchy--Schwarz inequality and Lemmas~\ref{primul} and \ref{primu2},}
	\begin{align*}
		\mathbb{E}\sum_{i=1}^N|\widetilde{X}^{i,N}_{k\theta}|^{2\varkappa+1}\leq N\widetilde{C}_{2\varkappa+1},
	\end{align*}
	and
	\begin{align*}
		\mathbb{E}\sum_{i=1}^N|\widetilde{X}^{i,N}_t-\widetilde{X}_{k\theta}^{i,N}|&\leq\sqrt{N\mathbb{E}\sum_{i=1}^N|\widetilde{X}^{i,N}_t-\widetilde{X}_{k\theta}^{i,N}|^2}\\
		&\leq N\sqrt{\theta}\sqrt{2d+16K^2+6K_G^2 +16K^2\widetilde{C}_2+12K_G^2\widetilde{C}_2}.
	\end{align*}
	Therefore, for any $\beta\in(0,1]$,
	\begin{align}\label{H-est-R}
		\mathbb{E}|\widetilde{R}_t^\varepsilon|
		\leq&e^{\widetilde{K}_\varepsilon(t-k\theta)}N(\Gamma_3\theta^{\alpha}+\Gamma_4\theta^{\frac{\beta}{2}}+\Gamma_5\theta^{\frac{1}{2}})
	\end{align}
	where
	\begin{align*}
		\Gamma_{3}&=(\|h'\|_\infty+\varepsilon)(2K+2K\widetilde{C}_{2l+2\beta-1}),\\
		\Gamma_4&=L(\|h'\|_\infty+\varepsilon)[1+2(\widetilde{C}_{\frac{2l}{2-\beta}})^{1-\frac{\beta}{2}}]\\
		&\quad\times(2d+16K^2+6K_G^2 +16K^2\widetilde{C}_2+12K_G^2\widetilde{C}_2)^{\frac{\beta}{2}},\\
		\Gamma_5&=(\|h'\|_\infty+\varepsilon)(\|\nabla_{y}G\|_{\infty}+\|\nabla_{x}G\|_{\infty})\\
		&\quad\times\sqrt{2d+16K^2+6K_G^2 +16K^2\widetilde{C}_2+12K_G^2\widetilde{C}_2}.
	\end{align*}

	Then, combining \eqref{H-est-1}--\eqref{H-est-R}, we have
	\begin{align*}
		&\mathbb{E}e^{\widetilde{K}_\varepsilon\theta}\sum_{i=1}^{N}h_\varepsilon(|X^{i,N}_{(k+1)\theta}-\widetilde{X}_{(k+1)\theta}^{i,N}|)\\
		\leq&\mathbb{E}\sum_{i=1}^Nh_\varepsilon(|X^{i,N}_{k\theta}-\widetilde{X}_{k\theta}^{i,N}|)
\\
		&+\mathbb{E}\int_{k\theta}^{(k+1)\theta}e^{\widetilde{K}_\varepsilon (s-k\theta)}\sum_{i=1}^N[1-\lambda_\delta^2(r_s^i)][r_s^i+\beta^+_0(r_s^i)h'(r_s^i)]ds
		\\&+N(\Gamma_3\theta^{\alpha}+\Gamma_4\theta^{\frac{\beta}{2}}+\Gamma_5\theta^{\frac{1}{2}})\int_{k\theta}^{(k+1)\theta}e^{\widetilde{K}_\varepsilon(s-k\theta)}ds\\
		=&\mathbb{E}\sum_{i=1}^Nh_\varepsilon(|X^{i,N}_{k\theta}-\widetilde{X}_{k\theta}^{i,N}|)
\\
		&+\Lambda^{k,\delta}+\frac{N}{\widetilde{K}_\varepsilon}(\Gamma_3\theta^{\alpha}+\Gamma_4\theta^{\frac{\beta}{2}}+\Gamma_5\theta^{\frac{1}{2}})(e^{\widetilde{K}_\varepsilon\theta}-1),
	\end{align*}
	where
	\[
\begin{aligned}
	\Lambda^{k,\delta}&:=\mathbb{E}\int_{k\theta}^{(k+1)\theta}e^{\widetilde{K}_\varepsilon(s-k\theta)}\\
		&\qquad\times\sum_{i=1}^N[1-\lambda_\delta^2(r_s^i)][r_s^i+\beta^+_0(r_s^i)h'(r_s^i)]ds.
\end{aligned}
	\]
	Similarly to the argument in Proposition \ref{singlee}, for each fixed $k$, it holds that 
	\begin{align*}
		\Lambda^{k,\delta}\le e^{\widetilde{K}_\varepsilon\theta}N\delta\theta(1+M^+\|h'\|_\infty)\to0,\,\text{as}\, \delta\to0.
	\end{align*}

	By iteration,
	\begin{align*}
		\mathbb{E}\sum_{i=1}^{N}h_\varepsilon(|X^{i,N}_{(k+1)\theta}-\widetilde{X}_{(k+1)\theta}^{i,N}|)\\
		\leq& e^{-(k+1)\widetilde{K}_\varepsilon\theta}\mathbb{E}\sum_{i=1}^{N}h_\varepsilon(|X^{i,N}_{0}-\widetilde{X}_{0}^{i,N}|)\\
		&+\sum_{j=0}^{k}\Lambda^{j,\delta}e^{-(k-j)\widetilde{K}_\varepsilon\theta}\\
		&+\frac{N(\Gamma_3\theta^{\alpha}+\Gamma_4\theta^{\frac{\beta}{2}}+\Gamma_5\theta^{\frac{1}{2}})}{\widetilde{K}_\varepsilon}.
	\end{align*}
	Then,
	\begin{align}\label{coll}
		W_{1,d_{l^1}}(\mu_{X^N_{(k+1)\theta}},\mu_{\widetilde{X}^N_{(k+1)\theta}})\notag\\
		\leq&\varlimsup\limits_{\delta\rightarrow0}\sum_{i=1}^{N}\mathbb{E}|X^{i,N}_{(k+1)\theta}-\widetilde{X}_{(k+1)\theta}^{i,N}|\notag\\
		\leq&\sup_{r>0}\frac{r}{h(r)+\varepsilon r}\varlimsup\limits_{\delta\rightarrow0}\sum_{i=1}^{N}\mathbb{E}h_\varepsilon(|X^{i,N}_{(k+1)\theta}-\widetilde{X}_{(k+1)\theta}^{i,N}|)\notag\\
		\leq&\sup_{r>0}\frac{r}{h(r)+\varepsilon r}e^{-(k+1)\widetilde{K}_\varepsilon\theta}\mathbb{E}\sum_{i=1}^{N}h_\varepsilon(|X^{i,N}_{0}-\widetilde{X}_{0}^{i,N}|)\notag\\
		&+\sup_{r>0}\frac{r}{h(r)+\varepsilon r}\frac{N(\Gamma_3\theta^{\alpha}+\Gamma_4\theta^{\frac{\beta}{2}}+\Gamma_5\theta^{\frac{1}{2}})}{\widetilde{K}_\varepsilon}\notag\\
		\leq&\sup_{r>0}\frac{r}{h(r)+\varepsilon r}\sup_{r>0}\frac{h(r)+\varepsilon r}{r}e^{-(k+1)\widetilde{K}_\varepsilon\theta}\notag\\
		&\qquad\times\mathbb{E}\sum_{i=1}^{N}|X^{i,N}_{0}-\widetilde{X}_{0}^{i,N}|\notag\\
		&+\sup_{r>0}\frac{r}{h(r)+\varepsilon r}\frac{N(\Gamma_3\theta^{\alpha}+\Gamma_4\theta^{\frac{\beta}{2}}+\Gamma_5\theta^{\frac{1}{2}})}{\widetilde{K}_\varepsilon}\notag\\
		=&A_\varepsilon e^{-(k+1)\widetilde{K}_\varepsilon\theta}\mathbb{E}\sum_{i=1}^{N}|X^{i,N}_{0}-\widetilde{X}_{0}^{i,N}|\notag\\
		&+\sup_{r>0}\frac{r}{h(r)+\varepsilon r}\frac{N(\Gamma_3\theta^{\alpha}+\Gamma_4\theta^{\frac{\beta}{2}}+\Gamma_5\theta^{\frac{1}{2}})}{\widetilde{K}_\varepsilon}.
	\end{align}
	Since $X^{i,N}_0=\widetilde{X}^{i,N}_0$ for $i=1,\cdots,N,$ we have
	\begin{align*}
		W_{1,d_{l^1}}(\mu_{X^N_{(k+1)\theta}},\mu_{\widetilde{X}^N_{(k+1)\theta}})\\
		\leq&\sup_{r>0}\frac{r}{h(r)+\varepsilon r}\frac{N(\Gamma_3\theta^{\alpha}+\Gamma_4\theta^{\frac{\beta}{2}}+\Gamma_5\theta^{\frac{1}{2}})}{\widetilde{K}_\varepsilon}.
	\end{align*}
	
	More generally, if $T\in[k\theta,(k+1)\theta],$ the same argument on $[k\theta,T]$ yields
	\begin{equation*}
		W_{1,d_{l^1}}(\mu_{X^N_{T}},\mu_{\widetilde{X}^N_{T}})\leq\sup\limits_{r>0}\frac{r}{h(r)+\varepsilon r}\frac{N(\Gamma_3\theta^{\alpha}+\Gamma_4\theta^{\frac{\beta}{2}}+\Gamma_5\theta^{\frac{1}{2}})}{\widetilde{K}_\varepsilon},
	\end{equation*}
	because $T-k\theta\leq\theta$ and the one-step estimate is unchanged.
\end{proof}

	For the later empirical-measure and path-functional arguments, we also establish a bound before sending the cutoff parameter to zero. Put
	\[
	S_\varepsilon:=\sup_{r>0}\frac{r}{h(r)+\varepsilon r},
	\qquad c_h:=1+M^+\|h'\|_\infty,
	\qquad D_\varepsilon:=\frac{S_\varepsilon c_h}{\widetilde K_\varepsilon}.
	\]
	The constant $c_h$ is independent of both cutoff and time step. Write $\mathbb E_\delta$ for expectation under the coupling \eqref{copling-par} with a fixed $\delta>0$ and identical initial data. If
	\[
	u_\delta(t):=\frac1N\mathbb E_\delta\sum_{i=1}^N h_\varepsilon(r_t^i),
	\qquad
	F_\theta:=\Gamma_3\theta^\alpha+\Gamma_4\theta^{\beta/2}
	+\Gamma_5\theta^{1/2},
	\]
	then \eqref{H-est-1}--\eqref{H-est-R} imply, in integrated form on successive grid intervals,
	\[
	u_\delta'(t)\le-\widetilde K_\varepsilon u_\delta(t)
	+F_\theta+\delta c_h,\qquad u_\delta(0)=0.
	\]
	Gronwall's inequality, without discarding the exponential weight on a grid interval, gives
	\begin{equation}\label{eq:finite-cutoff-disc}
	\frac1N\mathbb E_\delta\sum_{i=1}^N
	|\widetilde X_t^{i,N}-X_t^{i,N}|
	\le \mathcal E_{\rm disc}^{\rm MV}(\theta)+\delta D_\varepsilon,
	\qquad t\ge0.
	\end{equation}
	This estimate is uniform in $t,N,\theta$, with $\varepsilon$ fixed. In particular, for every $T>0$,
	\begin{equation}\label{eq:finite-cutoff-path}
	\frac1{NT}\int_0^T\mathbb E_\delta\sum_{i=1}^N
	|\widetilde X_t^{i,N}-X_t^{i,N}|\,dt
	\le \mathcal E_{\rm disc}^{\rm MV}(\theta)+\delta D_\varepsilon.
	\end{equation}
	Thus no exchange of a cutoff limit with a time integral is needed.

\begin{lem}[\blue{Contraction toward the invariant distribution}]\label{multi1}
	\blue{Under the hypotheses and notations of Theorem~\ref{mainre2}, for any $\varepsilon>0$ such that $\widetilde{K}_\varepsilon>0$ and any $T>0$, we have}
	\begin{align*}
		W_{1,d_{l^1}}(\mu_{X^N_{T}},\pi^N)\leq e^{-\widetilde{K}_\varepsilon T}A_\varepsilon W_{1,d_{l^1}}(\mu_{X^N_{0}},\pi^N).
	\end{align*}
	
\end{lem}

\begin{proof}
	Let $Y_t^N=(Y_t^{1,N},\ldots,Y_t^{N,N})$ be a stationary solution of \eqref{multiparticle} with law $\pi^N$, and couple $X_t^N$ and $Y_t^N$ by the approximate componentwise reflection coupling used in \cite{LWZ21}. Choose an optimal $d_{l^1}$-coupling of the initial laws, independently of the driving Brownian motions, and keep this initial coupling fixed as $\delta$ varies. With
	\[
	r_t^i=|X_t^{i,N}-Y_t^{i,N}|,\qquad
	H_t^\varepsilon=e^{\widetilde K_\varepsilon t}\sum_{i=1}^Nh_\varepsilon(r_t^i),
	\]
	\blue{the calculation is the same as in Proposition~\ref{multi}, except that no discretization remainder $\widetilde R_t^\varepsilon$ appears. More precisely, using the dissipativity inequality \eqref{dissipativity}, the interaction estimate controlled by $\|\nabla_yG\|_\infty$, and the definition of $\widetilde K_\varepsilon$, It\^o's formula gives}
	\[
	dH_t^\varepsilon\le dM_t^\varepsilon
	+e^{\widetilde K_\varepsilon t}\sum_{i=1}^N(1-\lambda_\delta^2(r_t^i))
	[r_t^i+\beta_0^+(r_t^i)h'(r_t^i)]\,dt .
	\]
	\blue{As in the proof of Proposition~\ref{multi}, the last term is bounded by}
	$e^{\widetilde K_\varepsilon t}N\delta(1+M^+\|h'\|_\infty)$ and therefore vanishes after integration when $\delta\downarrow0$. Hence
	\[
	\limsup_{\delta\downarrow0}\mathbb E_\delta
	\sum_{i=1}^Nh_\varepsilon(r_T^i)
	\le e^{-\widetilde K_\varepsilon T}\mathbb E\sum_{i=1}^Nh_\varepsilon(r_0^i).
	\]
	Using
	\[
	d_{l^1}(x,y)\le \sup_{r>0}\frac{r}{h(r)+\varepsilon r}
	\sum_{i=1}^Nh_\varepsilon(|x^i-y^i|)
	\]
	and, for the initial coupling,
	\[
	\sum_{i=1}^Nh_\varepsilon(|x^i-y^i|)
	\le \sup_{r>0}\frac{h(r)+\varepsilon r}{r}\,d_{l^1}(x,y),
	\]
	we obtain
	\[
	W_{1,d_{l^1}}(\mu_{X_T^N},\pi^N)
	\le A_\varepsilon e^{-\widetilde K_\varepsilon T}
	W_{1,d_{l^1}}(\mu_{X_0^N},\pi^N),
	\]
	where $A_\varepsilon=\sup_{r>0}\frac{r}{h(r)+\varepsilon r}
	\sup_{r>0}\frac{h(r)+\varepsilon r}{r}$. This is the asserted bound.
\end{proof}

\begin{proof}[\bf Proof of Theorem \ref{mainre2}]
	By the triangle inequality,
	\begin{equation*}
		W_{1,d_{l^1}}(\mu_{\widetilde{X}^N_{T}},\pi^N)\leq W_{1,d_{l^1}}(\mu_{X^N_{T}},\mu_{\widetilde{X}^N_{T}})+W_{1,d_{l^1}}(\mu_{X^N_{T}},\pi^N).
	\end{equation*}
	\blue{The desired result then follows by applying Proposition~\ref{multi} and Lemma~\ref{multi1}.}
\end{proof}
\begin{proof}[\bf Proof of Theorem \ref{thm:IPSmainre2}]
	\blue{For $x=(x^1,\ldots,x^N)\in(\mathbb{R}^d)^N$, let $V^N(x)=1+\frac{1}{N}\sum_{i=1}^N|x^i|^2$, and let $\widetilde{P}_{\theta,N}$ denote the one-step transition kernel of the Markov chain $\left\{(\widetilde{X}_{k\theta}^{1,N},\ldots,\widetilde{X}_{k\theta}^{N,N})\right\}_{k\in\mathbb{N}}$. Then}
	\begin{align*}
		&\widetilde{P}_{\theta,N}V^N(x^1,\cdots,x^N)\\
		=&\mathbb{E}[V^N(\widetilde{X}_\theta^{1,N},\cdots,\widetilde{X}_\theta^{N,N})]
		=\frac{1}{N}\sum_{i=1}^N(\mathbb{E}|\widetilde{X}^{i,N}_\theta|^2)+1\\
		=&\frac{1}{N}\sum_{i=1}^N\big(|x^i|^2+2\theta\langle x^i,b_\alpha(x^i)+\frac{1}{N-1}\sum_{j\neq i}G(x^i,x^j)\rangle\\&\qquad+\theta^2|b_\alpha(x^i)+\frac{1}{N-1}\sum_{j\neq i}G(x^i,x^j)|^2+2\theta d\big)+1.
	\end{align*}
	\blue{From \eqref{prop1} and \eqref{Glip},}
	\begin{align*}
		&\theta^2|b_\alpha(x^i)+\frac{1}{N-1}\sum_{j\neq i}G(x^i,x^j)|^2\\
\leq & 2\theta^2|b_{\alpha}(x^i)|^2+\frac{2\theta^2}{(N-1)^2}|\sum\limits_{j\neq i}G(x^i,x^j)|^2\\
		\leq&2\theta^2|b_{\alpha}(x^i)|^2+\frac{2\theta^2}{N-1}\sum\limits_{j\neq i}|G(x^i,x^j)|^2\\
		\leq&16K^2\theta^{2-2\alpha}(1+|x^i|^{2})+6K_G^2\theta^{2}\frac{1}{N-1}\sum_{j\neq i}(1+|x^i|^2+|x^j|^2)\\
		=&16K^2\theta^{2-2\alpha}+6K_G^2\theta^2+(16K^2\theta^{2-2\alpha}+6K_G^2\theta^2)|x^i|^2+6K_G^2\theta^2\frac{1}{N-1}\sum_{j\neq i}|x^j|^2.
	\end{align*}
	
	\blue{By Assumption {\bf (A2)},}
	\begin{align*}
		2\theta \langle x^i,b_\alpha(x^i)\rangle\le 2\theta(-c_1|x^i|^2+c_2).
	\end{align*}
	
	For any $\lambda>0$, Young's inequality yields
	\begin{align*}
		&2\theta\left\langle x^i,\frac{1}{N-1}\sum_{j\neq i}G(x^i,x^j)\right\rangle \\
		&=\frac{2\theta}{N-1}\sum_{j\neq i}
		\left\langle x^i,G(x^i,x^j)-G(0,x^j)+G(0,x^j)\right\rangle \\
		&\le 2\theta\|\nabla_xG\|_\infty |x^i|^2
		+\frac{\theta}{N-1}\sum_{j\neq i}
		\left(\lambda |x^i|^2+\lambda^{-1}|G(0,x^j)|^2\right) \\
		&\le \theta\left(2\|\nabla_xG\|_\infty+\lambda\right)|x^i|^2
		+\frac{2\theta K_G^2}{(N-1)\lambda}
		\sum_{j\neq i}(1+|x^j|^2).
	\end{align*}
	Consequently,
	\begin{align*}
		&\widetilde P_{\theta,N}V^N(x)\\
		\le{}& \frac1N\sum_{i=1}^N
		\Big[1-2c_1\theta+2\theta\|\nabla_xG\|_\infty
		+\lambda\theta+\frac{2\theta K_G^2}{\lambda}
		+16K^2\theta^{2-2\alpha}+12K_G^2\theta^2\Big]|x^i|^2 \\
		&+2\theta d+2\theta c_2+\frac{2\theta K_G^2}{\lambda}
		+16K^2\theta^{2-2\alpha}+6K_G^2\theta^2+1.
	\end{align*}
	Taking $\lambda=\sqrt2K_G$, the coefficient of $\frac1N\sum_{i=1}^N|x^i|^2$ becomes
	$1-2\theta(c_1-\|\nabla_xG\|_\infty-\sqrt2K_G
	-8K^2\theta^{1-2\alpha}-6K_G^2\theta)$.

	Hence, when $\theta<\widetilde{\eta},$ we have 
	\begin{align*}
		\widetilde{P}_{\theta,N}V^N(x^1,\cdots,x^N)\le\widetilde{r}_1V^N(x^1,\cdots,x^N)+C,
	\end{align*}
	where $\widetilde r_1$ is given by \eqref{wr1}. The argument used in the proof of Theorem~\ref{thm:mainre2} shows that the $N$-particle chain $(\widetilde X_{k\theta}^N)_{k\in\mathbb N}$ has a unique invariant distribution. Moreover, there exist constants $\widetilde\rho\in(0,1)$ and $C>0$ such that, for every $k\in\mathbb N$,
	\begin{align*}
		\sup\limits_{|f|\le V^N}|\widetilde{P}^k_{\theta,N}f(x^1,\cdots,x^N)-\pi^{\rm TEM,N}_\theta(f)|\le C\widetilde{\rho}^kV^N(x^1,\cdots,x^N).
	\end{align*}
	\blue{For any function $f$ that is $1$-Lipschitz with respect to $d_{l^1}$ and satisfies $f(0,\ldots,0)=0$,}
	\begin{align*}
		{\color{black}\sup_{x,y\in(\mathbb{R}^d)^N,\,x\neq y}\frac{|f(x)-f(y)|}{d_{l^1}(x,y)}\le1,}
	\end{align*}
 we have
	\begin{align*}	&|f(x^1,\cdots,x^N)|\le\sum_{i=1}^N|x^i|\le NV^N(x^1,\cdots,x^N).
	\end{align*}
	\blue{Thus $f/N$ is bounded in absolute value by $V^N$, and the $V^N$-uniform ergodic bound gives}
	\begin{align*}
		&|\mu_0^{\otimes N}\widetilde{P}^k_{\theta,N}(f)-\pi^{\rm TEM,N}_\theta(f)|\notag\\
		=&{\color{black}N\left|\mu_0^{\otimes N}\widetilde{P}^k_{\theta,N}(f/N)-\pi^{\rm TEM,N}_\theta(f/N)\right|\le CN\widetilde{\rho}^k\mu_0^{\otimes N}(V^N).}
	\end{align*}
	Kantorovich--Rubinstein duality yields
	\begin{align}\label{eq:coro0}
		&W_{1,d_{l^1}}(\mu_{\widetilde{X}^N_{k\theta}},\pi^{\rm TEM,N}_\theta)\notag\\
		&\qquad\le CN\widetilde{\rho}^k\mu_0^{\otimes N}(V^N).
	\end{align}
	By the triangle inequality and Theorem \ref{mainre2}, for any $k\in\mathbb{N},$ we have
	\begin{align*}
		W_{1,d_{l^1}}(\pi^{\rm TEM,N}_\theta,\pi^N)\le& W_{1,d_{l^1}}(\pi^{\rm TEM,N}_\theta,\mu_{\widetilde{X}^N_{k\theta}})+W_{1,d_{l^1}}(\mu_{\widetilde{X}^N_{k\theta}},\pi^N)\\
        \le&CN\widetilde{\rho}^k\mu_0^{\otimes N}(V^N)+\mathcal E_{\rm relax}^{N}(k\theta)+\mathcal E_{\rm disc}^{N}(\theta).
	\end{align*}
	By Lemma \ref{primul}, $\mu_0^{\otimes N}(V^N)<\infty.$ Letting $k\to\infty$ completes the proof. 
\end{proof}

\subsection{Proof of the Main Result for the McKean--Vlasov Equation}

\begin{bluechange}
By Remark~\ref{rem:LWZ-general-drift}, the continuous-time contraction and propagation-of-chaos estimates of \cite{LWZ21} apply to the present non-gradient drift. We use them in the form obtained by approximate componentwise reflection coupling. The interaction is normalized by $1/(N-1)$, as in \cite{LWZ21}.
\end{bluechange}

\begin{lem}[McKean--Vlasov contraction]\label{lem:MV-contraction}
	Under the assumptions and notations of Theorem \ref{MVTHM}, the adapted contraction estimate \cite[Corollary 2.9]{LWZ21} gives, for every $T\ge0$,
	\begin{align}\label{firstterm}
		W_1(\pi^{MV},\mu_{\bar{X}_T})\le A_\varepsilon e^{-\widetilde{K}_\varepsilon T}W_1(\pi^{MV},\mu_0).
	\end{align}
\end{lem}

\begin{lem}[Uniform propagation of chaos]\label{lem:POC}
	Under the assumptions and notations of Theorem \ref{MVTHM}, the coupling proof of the uniform propagation-of-chaos estimate \cite[Theorem 2.11(b), Eq. (4.17)]{LWZ21} yields, for each $1\le i\le N$ and every $T\ge0$,
	\begin{align}\label{secondterm}
		W_1(\mu_{\bar{X}_T},\mu_{X_T^{i,N}})\leq\frac{1}{\sqrt{N-1}}\frac{A_\varepsilon}{\widetilde{K}_\varepsilon}\|\nabla_{y}G\|_\infty \sup\limits_{t\ge0}\sqrt{\mathbb{E}|\bar{X}_t|^2}.
	\end{align}
\end{lem}

\begin{proof}[\bf Proof of Theorem~\ref{MVTHM}]
	By the triangle inequality, for any $1\leq i \leq N$
	\begin{align}\label{triangle}
		W_1(\pi^{MV},\mu_{\widetilde{X}_T^{i,N}})\leq& W_1(\pi^{MV},\mu_{\bar{X}_T})+W_1(\mu_{\bar{X}_T},\mu_{X_T^{i,N}})+W_1(\mu_{X_T^{i,N}},\mu_{\widetilde{X}_T^{i,N}}).
	\end{align}
	{It remains to control the last term in \eqref{triangle}. In the coupling constructed in Lemma \ref{multi}, choose the initial coupling so that $X_0^{j,N}=\widetilde X_0^{j,N}$ for every $j$. The construction is exchangeable in the particle index; hence, for each fixed $i$,
		\[
		\mathbb{E}h_\varepsilon(r_T^i)=\frac1N\mathbb{E}\sum_{j=1}^Nh_\varepsilon(r_T^j).
		\]
		Repeating the argument leading to \eqref{coll}, dividing the resulting bound by $N$, and using the identical initial data, gives}
	\begin{equation}\label{multi-marginal}
		W_1(\mu_{X_T^{i,N}},\mu_{\widetilde{X}_T^{i,N}})\leq \sup\limits_{r>0}\frac{r}{h(r)+\varepsilon r}\frac{\Gamma_3\theta^{\alpha}+\Gamma_4\theta^{\frac{\beta}{2}}+\Gamma_5\theta^{\frac{1}{2}}}{\widetilde{K}_\varepsilon}.
	\end{equation}
	By \eqref{triangle}, Lemma \ref{lem:MV-contraction}, Lemma \ref{lem:POC}, and \eqref{multi-marginal}, we obtain the desired result.
\end{proof}

\begin{proof}[\bf Proof of Corollary \ref{coro0}]
	\blue{For two probability measures $\mu$ and $\nu$ on $(\mathbb{R}^d)^N$, we have}
	\begin{align*}
		W_{1,d_{l^1}}(\mu,\nu)\ge\sum_{i=1}^NW_1(\mu^i,\nu^i),
	\end{align*}
	\blue{where $\mu^i$ and $\nu^i$ denote the $i$-th marginals of $\mu$ and $\nu$, respectively.}
	It follows by Theorem \ref{MVTHM} that for any $k\in\mathbb{N},$
	\begin{align*}
		W_1(\pi_\theta^{\rm TEM,N,1},\pi^{MV})\le &W_1(\mu_{\widetilde{X}^{1,N}_{k\theta}},\pi^{MV})+W_1(\mu_{\widetilde{X}^{1,N}_{k\theta}},\pi_\theta^{\rm TEM,N,1})\\
		\le& \mathcal E_{\rm relax}^{\rm MV}(k\theta)+\mathcal E_{\rm disc}^{\rm MV}(\theta)+\mathcal E_{\rm part}^{\rm MV}(N)+W_{1,d_{l^1}}(\mu_{\widetilde{X}^N_{k\theta}},\pi^{\rm TEM,N}_\theta)\\
		\le&\mathcal E_{\rm relax}^{\rm MV}(k\theta)+\mathcal E_{\rm disc}^{\rm MV}(\theta)+\mathcal E_{\rm part}^{\rm MV}(N)+CN\widetilde{\rho}^k\mu_0^{\otimes N}(V^N),
	\end{align*}
	\blue{where the last inequality follows from \eqref{eq:coro0}.}
	Letting $k\to\infty$ completes the proof.
\end{proof}

\begin{proof}[\bf Proof of  Corollary~\ref{coro1}]
Write $\nu_j:=\mathcal L(\widetilde X_{j\theta}^{1,N})$ and
$D:=W_1(\mu_0,\pi^{MV})$. Exchangeability gives
$\mathbb E[\mu_{N,k}]=k^{-1}\sum_{j=0}^{k-1}\nu_j$. By
Kantorovich--Rubinstein duality, $W_1$ is convex in its first argument,
so
\[
 W_1\bigl(\mathbb E[\mu_{N,k}],\pi^{MV}\bigr)
 \leq \frac1k\sum_{j=0}^{k-1}W_1(\nu_j,\pi^{MV}).
\]
For $j\geq1$, Theorem~\ref{MVTHM} bounds the summand by
$A_\varepsilon e^{-\widetilde K_\varepsilon j\theta}D
+\mathcal E_{\rm disc}^{\rm MV}(\theta)
+\mathcal E_{\rm part}^{\rm MV}(N)$.
The same bound holds at $j=0$ because $\nu_0=\mu_0$,
$A_\varepsilon\geq1$, and both error terms are nonnegative. Hence the
geometric-series identity yields
\begin{align*}
 W_1\bigl(\mathbb E[\mu_{N,k}],\pi^{MV}\bigr)
 &\leq \mathcal E_{\rm disc}^{\rm MV}(\theta)
       +\mathcal E_{\rm part}^{\rm MV}(N)
       +\frac{A_\varepsilon D}{k}
         \sum_{j=0}^{k-1}e^{-\widetilde K_\varepsilon j\theta}\\
 &=\mathcal E_{\rm disc}^{\rm MV}(\theta)
       +\mathcal E_{\rm part}^{\rm MV}(N)
       +A_\varepsilon D
         \frac{1-e^{-\widetilde K_\varepsilon T}}
              {k(1-e^{-\widetilde K_\varepsilon\theta})}.
\end{align*}
Since $x/(1-e^{-x})\leq1+x$ for $x>0$, $k=T/\theta$, and
$C_{\rm erg}=A_\varepsilon D/\widetilde K_\varepsilon$, \eqref{eq:coro1} follows. 
\end{proof}

\begin{proof}[\bf Proof of Corollary \ref{coro2}]
	Let $\bar{X}^1_T,\cdots,\bar{X}^N_T$ be independent copies of \eqref{mckeanvlasov} at time $T$. By the triangle inequality
	\begin{align*}
		W_1\left(\frac{1}{N}\sum_{i=1}^N\delta_{\widetilde{X}_T^{i,N}},\pi^{MV}\right)\leq&	W_1\left(\frac{1}{N}\sum_{i=1}^N\delta_{\widetilde{X}_T^{i,N}},\frac{1}{N}\sum_{i=1}^N\delta_{X_T^{i,N}}\right)\\
		&+ W_1\left(\frac{1}{N}\sum_{i=1}^N\delta_{X_T^{i,N}},\frac{1}{N}\sum_{i=1}^N\delta_{\bar{X}_T^{i}}\right)\\
		&+W_1\left(\frac{1}{N}\sum_{i=1}^N\delta_{\bar{X}_T^{i}},\mu_{\bar{X}_T}\right)+W_1\left(\mu_{\bar{X}_T},\pi^{MV}\right)\\
		:=&I_1+I_2+I_3+I_4.
	\end{align*}
	For $I_1$, \eqref{eq:finite-cutoff-disc} yields
	\begin{align}\label{I1}
		\mathbb{E}I_1\leq\mathbb{E}[\frac{1}{N}\sum_{i=1}^{N}|\widetilde{X}_T^{i,N}-X_T^{i,N}|]\le\mathcal E_{\rm disc}^{\rm MV}(\theta)+\delta D_\varepsilon,
	\end{align}

	For $I_2$, by \cite[Proposition 4.3]{LWZ21}
	\begin{align}\label{I2}	
		\mathbb{E}I_2\leq\frac{1}{\sqrt{N-1}}\frac{A_\varepsilon}{\widetilde{K}_\varepsilon}\|\nabla_yG\|_\infty\sup\limits_{t\ge0}\sqrt{\mathbb{E}|\bar{X}_t|^2}.
	\end{align}	
	Applying \cite[Theorem 1]{FG} with $p=1$ and $q=4$ for $d\in\{1,2\}$, and with $p=1$ and $q=2$ for $d\ge3$, we obtain
	\begin{align}\label{I3}
		\mathbb{E}I_3\leq C_d
		\begin{cases}
			\sup\limits_{t\geq0}[\mathbb{E}|\bar{X}_t|^4]^\frac{1}{4}[N^{-\frac{1}{2}}+N^{-\frac{3}{4}}], \,  d=1,\\
			\sup\limits_{t\geq0}[\mathbb{E}|\bar{X}_t|^4]^\frac{1}{4}[N^{-\frac{1}{2}}\ln(1+N)+N^{-\frac{3}{4}}], \, d=2,\\
			\sup\limits_{t\geq0}[\mathbb{E}|\bar{X}_t|^2]^\frac{1}{2}[N^{-\frac{1}{d}}+N^{-\frac{1}{2}}], \,  d\geq3.
		\end{cases}
	\end{align}	
	By \eqref{firstterm}, we have
	\begin{align}\label{I4}
		I_4\leq A_\varepsilon e^{-\widetilde{K}_\varepsilon T}W_1(\pi^{MV},\mu_0).
	\end{align}
	Combining \eqref{I1}-\eqref{I4}, as $\delta\to0,$ we obtain \eqref{W1}.
\end{proof}

\begin{proof}[\bf Proof of Corollary \ref{coro:concen}]
Let
\[
a_0:=a+4\left(\frac{C_{\rm poc}}{\sqrt{N-1}}+\frac{C_{\rm erg}}{T}\right)>0.
\]
We first assume that $f\in C^2(\mathbb R^d)$; the passage to a general $1$-Lipschitz function is given at the end of the proof.
We choose the version of $\widetilde{X}^{i,N}$ and $X^{i,N}$ given by the coupling \eqref{copling-par}. It follows that
    \begin{align*}
        &\mathbb{P}\big(\frac{1}{kN}\sum_{i=1}^N\sum_{j=0}^{k-1}f(\widetilde{X}_{j\theta}^{i,N})-\pi^{MV}(f)>a_0\big)\\
        \le&\mathbb{P}\big(\frac{1}{kN}\sum_{i=1}^N\sum_{j=0}^{k-1}f(\widetilde{X}_{j\theta}^{i,N})-\frac{1}{TN}\sum_{i=1}^N\int_0^Tf(\widetilde{X}_{t}^{i,N})dt>a_0/4\big)\\
        &+\mathbb{P}\big(\frac{1}{NT}\sum_{i=1}^N\int_0^Tf(\widetilde{X}_{t}^{i,N})dt-\frac{1}{TN}\int_0^T\sum_{i=1}^Nf(X_{t}^{i,N})dt>a_0/4\big)\\
        &+\mathbb{P}\big(\frac{1}{TN}\int_0^T\sum_{i=1}^Nf(X_{t}^{i,N})dt-\frac{1}{TN}\int_0^T\sum_{i=1}^N\mathbb{E}[f(X_{t}^{i,N})]dt>a_0/4\big)\\
        &+\mathbb{P}\big(\frac{1}{TN}\int_0^T\sum_{i=1}^N\mathbb{E}[f(X_{t}^{i,N})]dt-\pi^{MV}(f)>a_0/4\big)\\
        :=&I_1+I_2+I_3+I_4.
    \end{align*}
    For the first term, from Lemma \ref{primu2}, we have
    \begin{align*}
        I_1\le&\frac{4}{a_0}\mathbb{E}|\frac{1}{TN}\sum_{i=1}^{N}\sum_{j=0}^{k-1}\int_{j\theta}^{(j+1)\theta}\big[f(\widetilde{X}_{j\theta}^{i,N})-f(\widetilde{X}_{t}^{i,N})\big]dt|\\
        \le&\frac{4}{a_0}\frac{1}{TN}\sum_{i=1}^{N}\sum_{j=0}^{k-1}\int_{j\theta}^{(j+1)\theta}\mathbb{E}|\widetilde{X}_{j\theta}^{i,N}-\widetilde{X}_{t}^{i,N}|dt
        \le\frac{4}{a_0}\sqrt{\theta}C_{\rm mom}.
    \end{align*}
    For the second term, from \eqref{eq:finite-cutoff-path},
    \begin{align*}
        I_2\le\frac{4}{a_0}\frac{1}{NT}\int_0^T\sum_{i=1}^N\mathbb{E}|\widetilde{X}_t^{i,N}-X_t^{i,N}|dt\le\frac{4}{a_0}(\mathcal E_{\rm disc}^{\rm MV}(\theta)+\delta D_\varepsilon).
    \end{align*}
    
    For the third term, put $g_N(x)=N^{-1}\sum_{i=1}^N f(x^i)$,
    $c_0=1-\|\nabla_yG\|_\infty\|h'\|_\infty>0$, and
    $c_{\rm Lip}=h'(0)/c_0$. The integrated componentwise coupling
    estimate of \cite[Theorem~2.1]{LWZ21}, with the non-gradient
    adaptation described in Remark~\ref{rem:LWZ-general-drift}, gives
    \[
    \left|\int_0^s\big(P_r^N g_N(x)-P_r^N g_N(y)\big)\,dr\right|
    \leq\frac{1}{Nc_0}\sum_{i=1}^N h(|x^i-y^i|),\qquad s\geq0.
    \]
    Indeed, the coupling calculation uses \eqref{h}; the interaction
    contributes at most $\|\nabla_yG\|_\infty\|h'\|_\infty$ times the
    total coordinate distance. Let
\[
    u(t,x)=T^{-1}\int_0^{T-t}P_s^N g_N(x)\,ds
\]
and
    $F_T=T^{-1}\int_0^T g_N(X_s^N)\,ds$. Subdividing a displacement in
    one coordinate and using $nh(r/n)\to h'(0)r$ shows that
    $\operatorname{Lip}_i u(t,\cdot)\leq c_{\rm Lip}/(NT)$.

    For smooth coefficients, the backward equation and It\^o's formula
    yield $F_T-u(0,X_0^N)=M_T$, where
    \[
    M_T=\sqrt2\sum_{i=1}^N\int_0^T\nabla_i u(s,X_s^N)\cdot dB_s^i,
    \qquad
    \langle M\rangle_T\leq\frac{2c_{\rm Lip}^2}{NT}.
    \]
    Thus the factor $2$ from the noise $\sqrt2\,dB_t$ must be retained.
    Tensorizing \eqref{eq:concercg} for the independent initial values
    bounds the centered log-moment generating function of
    $u(0,X_0^N)$ by
    $\lambda^2c_G(\mu_0)c_{\rm Lip}^2/(2NT^2)$. The conditional
    exponential-martingale bound therefore gives
    \[
    \mathbb E\exp\{\lambda(F_T-\mathbb EF_T)\}
    \leq\exp\left\{\frac{\lambda^2c_{\rm Lip}^2}{2NT}
    \left(2+\frac{c_G(\mu_0)}T\right)\right\},\qquad\lambda>0.
    \]
    For the coefficients under (A1)--(A5), this calculation follows by
    smooth approximation and localization: common spatial shifts
    preserve \eqref{dissipativity}, while mollification does not enlarge
    $\|\nabla_yG\|_\infty$; the uniform bound passes to the limit.
    Chernoff's inequality at level $a_0/4$ now yields
    \[
    I_3\leq\exp\left(-\frac{a_0^2
    (1-\|\nabla_yG\|_\infty\|h'\|_\infty)^2}
    {32(h'(0))^2(2+c_G(\mu_0)/T)}NT\right).
    \]
    
For the fourth term, from Lemmas~\ref{lem:MV-contraction} and \ref{lem:POC}, we  have
	\begin{align}\label{eq:bias}
		&|\mathbb{E}\frac{1}{T}\int_0^Tf(X_t^{1,N})dt-\pi^{MV}(f)|\notag\\
		=&\left|\frac{1}{T}\int_0^T\left[\mathbb{E}f(X_t^{1,N})-\mu_{\bar{X}_t}(f)+\mu_{\bar{X}_t}(f)-\pi^{MV}(f)\right]dt\right|\notag\\
		\le&\frac{\|f\|_{\rm Lip}}{T}\int_0^T\left[W_1(\mu_{{X}_t^{1,N}},\mu_{\bar{X}_t})+W_1(\mu_{\bar{X}_t},\pi^{MV})\right]dt\notag\\
        \le&\|f\|_{\rm Lip}\big[\frac{1}{\sqrt{N-1}}\frac{A_\varepsilon}{\widetilde{K}_\varepsilon}\|\nabla_{y}G\|_\infty \sup\limits_{t\ge0}\sqrt{\mathbb{E}|\bar{X}_t|^2}+\frac{1}{T}\int_0^TA_\varepsilon e^{-\widetilde{K}_\varepsilon t}W_1(\pi^{MV},\mu_0)dt\big]\notag\\
        \le&\|f\|_{\rm Lip}\big[\frac{1}{\sqrt{N-1}}\frac{A_\varepsilon}{\widetilde{K}_\varepsilon}\|\nabla_{y}G\|_\infty \sup\limits_{t\ge0}\sqrt{\mathbb{E}|\bar{X}_t|^2}+\frac{A_\varepsilon}{\widetilde{K}_\varepsilon T} W_1(\pi^{MV},\mu_0)\big].
	\end{align}
  It follows that
    \begin{align*}
        &|\frac{1}{TN}\int_0^T\sum_{i=1}^N\mathbb{E}[f(X_{t}^{i,N})]dt-\pi^{MV}(f)|\\
		&=|\frac{1}{T}\int_0^T\mathbb{E}[f(X_{t}^{1,N})]dt-\pi^{MV}(f)|\\
        &\le\frac{C_{\rm poc}}{\sqrt{N-1}}+\frac{C_{\rm erg}}{T}.
    \end{align*}
Therefore,
\begin{align*}
    &\mathbb{P}\big(\frac{1}{TN}\int_0^T\sum_{i=1}^N\mathbb{E}[f(X_{t}^{i,N})]dt-\pi^{MV}(f)>a_0/4\big)\\
    =&\mathbb{P}\left(\begin{gathered}
    \frac{1}{TN}\int_0^T\sum_{i=1}^N\mathbb{E}[f(X_{t}^{i,N})]dt-\pi^{MV}(f)\\
    >a/4+\frac{C_{\rm poc}}{\sqrt{N-1}}+\frac{C_{\rm erg}}{T}
    \end{gathered}\right)=0.
\end{align*}

Hence, as $\delta\to0,$ it follows that
\begin{align*}
    &\mathbb{P}\left(\frac{1}{kN}\sum_{i=1}^N\sum_{j=0}^{k-1}f(\widetilde{X}_{j\theta}^{i,N})-\pi^{MV}(f)>a_0\right)\\
    &\quad\le\frac{4}{a_0}\left(\mathcal E_{\rm disc}^{\rm MV}(\theta)+C_{\rm mom}\sqrt{\theta}\right)
    \\&\quad+\exp\left(-\frac{a_0^2(1-\|\nabla_yG\|_\infty\|h'\|_\infty)^2}{32(h'(0))^2({2}+c_G(\mu_0)/T)}NT\right).
\end{align*}
To remove the temporary smoothness assumption, let $\rho$ be a standard nonnegative $C_c^\infty(\mathbb R^d)$ mollifier and set $\rho_n(x)=n^d\rho(nx)$ and $f_n=f*\rho_n$. Then $f_n\in C^\infty(\mathbb R^d)$, $\|f_n\|_{\rm Lip}\leq1$, and
\[
\|f_n-f\|_\infty\leq \frac{1}{n}\int_{\mathbb R^d}|z|\rho(z)\,dz\longrightarrow0.
\]
Consequently, the centered time averages corresponding to $f_n$ converge almost surely to that corresponding to $f$, while $\pi^{MV}(f_n)\to\pi^{MV}(f)$. The bound above is uniform in $n$; hence, by passing to the limit (equivalently, by applying the estimate with a deviation level decreased by an arbitrary positive number and then letting that number tend to zero), the same inequality holds for every $1$-Lipschitz $f$.

This completes the proof.
\end{proof}

\section{Appendix}
\subsection{\texorpdfstring{\blue{Moment estimates for the discretization of the Langevin diffusion}}{Moment estimates for the discretization of the Langevin diffusion}}
\begin{lem}\label{lem:gaussian-moment-explicit}
Let $r\ge2$, let $Z\sim N(0,I_d)$, and let $Y$ be an $\mathbb R^d$-valued
random vector independent of $Z$ with $\mathbb E|Y|^r<\infty$. Then, for
$s\ge0$,
\begin{equation}\label{eq:gaussian-Lr-square}
\bigl(\mathbb E|Y+\sqrt{2s}Z|^r\bigr)^{2/r}
\le \bigl(\mathbb E|Y|^r\bigr)^{2/r}+2(d+r-2)s.
\end{equation}
\end{lem}
\begin{proof}
For a deterministic $y\in\mathbb R^d$, set
$F_y(s):=\mathbb E|y+\sqrt2B_s|^r$. It\^o's formula gives, for $s>0$,
\begin{equation*}
\begin{aligned}
    |y+\sqrt{2}B_s|^r&=|y|^r+r(d+r-2)\int_0^s|y+\sqrt{2}B_u|^{r-2}du\\
		&\quad+\sqrt{2}\int_0^sr|y+\sqrt{2}B_u|^{r-2}\langle y+\sqrt{2}B_u,dB_u\rangle.
\end{aligned}
\end{equation*}
It follows that 
\[
F_y'(s)=r(d+r-2)\mathbb E|y+\sqrt2B_s|^{r-2}
\le r(d+r-2)F_y(s)^{1-2/r}.
\]
Consequently,
$\frac{d}{ds}F_y(s)^{2/r}\le2(d+r-2)$, and continuity at $s=0$ yields
\[
\mathbb E|y+\sqrt{2s}Z|^r
\le\bigl(|y|^2+2(d+r-2)s\bigr)^{r/2}.
\]
Conditioning on $Y$
and applying Minkowski's inequality in $L^{r/2}$, we obtain
\begin{align*}
\bigl(\mathbb E|Y+\sqrt{2s}Z|^r\bigr)^{2/r}\\
		&=\bigl(\mathbb{E}[\mathbb E\bigl[|Y+\sqrt{2s}Z|^r|Y]\bigr]\bigr)^{2/r}\\
&\le\bigl(\mathbb{E}(|Y|^2+2(d+r-2)s)^{r/2}\bigr)^{2/r}\\
&\le\bigl(\mathbb E|Y|^r\bigr)^{2/r}+2(d+r-2)s,
\end{align*}
which proves \eqref{eq:gaussian-Lr-square}.
\end{proof}
\begin{proof}[Proof of Lemma \ref{pri}]
Fix $r\in[2,q]$ and $0<\theta<\eta$. The linear-growth estimate
\begin{equation}\label{drift-cont}
|b_\alpha(x)|^2\le8K^2\theta^{-2\alpha}(1+|x|^2)
\end{equation}
and the finite initial $r$-th moment imply that all the moments used below
are finite at every fixed time. The definition of $\eta$ gives
\begin{equation}\label{r1}
r_1:=1-\kappa_\alpha\theta\in(0,1).
\end{equation}
For $0\le s\le\theta$, Assumption~\textbf{(A2)} and
\eqref{drift-cont} imply
\begin{align}\label{eq:Lan-deterministic-substep}
|x+s b_\alpha(x)|^2
&\le\bigl(1-2c_1s+8K^2s^2\theta^{-2\alpha}\bigr)|x|^2
       +2c_2s+8K^2s^2\theta^{-2\alpha}\notag\\
&\le(1-\kappa_\alpha s)|x|^2+D_{\rm L}s.
\end{align}
Indeed, $s^2\theta^{-2\alpha}\le s\theta^{1-2\alpha}$, and
$\theta^{1-2\alpha}\le1$. If $\alpha<1/2$, the step-size restriction
also gives $8K^2\theta^{1-2\alpha}\le c_1$; if $\alpha=1/2$, the
coefficient is bounded using $\kappa_\alpha=2(c_1-4K^2)$.
In particular, $1-\kappa_\alpha s\ge0$.

Set
\[
U_r(t):=\bigl(\mathbb E|\widetilde X_t|^r\bigr)^{2/r},
\qquad
H_{r,d}^{\rm L}:=\frac{D_{\rm L}+2(d+r-2)}{\kappa_\alpha}.
\]
The Brownian increment over $[k\theta,k\theta+s]$ is independent of
$\mathscr F_{k\theta}$. Thus \eqref{interpolation-single},
Lemma~\ref{lem:gaussian-moment-explicit},
\eqref{eq:Lan-deterministic-substep}, and Minkowski's inequality give
\begin{align}\label{eq:Lan-root-moment-recursion}
U_r(k\theta+s)
&\le\left(
\mathbb E
\bigl|\widetilde X_{k\theta}+s b_\alpha(\widetilde X_{k\theta})\bigr|^{r}
\right)^{2/r}+2(d+r-2)s\notag\\
&\le(1-\kappa_\alpha s)U_r(k\theta)
     +\bigl(D_{\rm L}+2(d+r-2)\bigr)s.
\end{align}
Taking $s=\theta$ and iterating this scalar inequality in the time index
$k$ yields
\[
U_r(k\theta)\le r_1^k m_r^{2/r}
                  +(1-r_1^k)H_{r,d}^{\rm L}
\le\max\{m_r^{2/r},H_{r,d}^{\rm L}\}.
\]
Applying \eqref{eq:Lan-root-moment-recursion} once more for
$0\le s\le\theta$ gives the same bound throughout each interpolation
interval, because its right-hand side is
$(1-\kappa_\alpha s)U_r(k\theta)
 +\kappa_\alpha s H_{r,d}^{\rm L}$.
The bound is uniform over $0<\theta<\eta$, proving
\eqref{eq:moment-Lan-explicit}. In particular,
\begin{equation}\label{pri1}
C_2\le\max\left\{m_2,
\frac{2c_2+8K^2+2d}{\kappa_\alpha}\right\}.
\end{equation}
The estimates for $0<r<2$ follow from Lyapunov's inequality.
\end{proof}

\begin{proof}[Proof of Lemma \ref{pri2}]
	Since
	\begin{equation*}
		\widetilde{X}_t-\widetilde{X}_{k\theta}=(t-k\theta)b_{\alpha}(\widetilde{X}_{k\theta})+\sqrt{2}({B}_{t}-{B}_{k\theta}),
	\end{equation*}
	by \eqref{prop1} and \eqref{drift-cont}, we obtain
	\begin{align*}
		&\mathbb{E}[|(t-k\theta)b_{\alpha}(\widetilde{X}_{k\theta})+\sqrt{2}({B}_{t}-{B}_{k\theta})|^2]\\
		=&(t-k\theta)^2\mathbb{E}[|b_{\alpha}(\widetilde{X}_{k\theta})|^2]+2(t-k\theta)d\\
		\leq&\theta^{2-2\alpha}8K^2\mathbb{E}(|\widetilde{X}_{k\theta}|^2+1)+2\theta d.
	\end{align*}
	By \eqref{pri1} and $0<\theta\le 1$, we have
	\begin{align*}
		\mathbb{E}[|\widetilde{X}_t-\widetilde{X}_{k\theta}|^2]=\mathbb{E}[|(t-k\theta)b_{\alpha}(\widetilde{X}_{k\theta})+\sqrt{2}(B_{t}-B_{k\theta})|^2]\leq\theta(8K^2(1+C_2)+2d).
	\end{align*}
\end{proof}
\subsection{\texorpdfstring{\blue{Moment estimates for the discretization of the interacting particle system}}{Moment estimates for the discretization of the interacting particle system}}
	\begin{proof}[Proof of Lemma \ref{primul}]
Fix $N\ge2$, $0<\theta<\widetilde\eta$, and $r\in[2,q]$.
For each fixed time, finiteness of the $r$-th moments follows from
\eqref{prop1}, \eqref{Glip}, and the finite initial $r$-th moment.
Exchangeability gives, for $j\ne1$ and all relevant $p>0$,
\begin{equation}\label{exchan}
\begin{aligned}
\mathbb E|\widetilde X_t^{j,N}|^p
 &=\mathbb E|\widetilde X_t^{1,N}|^p,\\
\mathbb E|G(\widetilde X_t^{1,N},\widetilde X_t^{j,N})|^2
 &=\mathbb E|G(\widetilde X_t^{1,N},\widetilde X_t^{2,N})|^2.
\end{aligned}
\end{equation}
Write $L_x:=\|\nabla_xG\|_\infty$. For
$x=(x^1,\ldots,x^N)$ and $0\le s\le\theta$, define
\[
\Phi_{\theta,s}^i(x):=
 x^i+s b_\alpha(x^i)+\frac{s}{N-1}\sum_{j\ne i}G(x^i,x^j).
\]
The growth estimates and Young's inequality give
\begin{equation}\label{eq:particle-deterministic-substep}
|\Phi_{\theta,s}^i(x)|^2
\le a_s|x^i|^2+\frac{b_s}{N-1}\sum_{j\ne i}|x^j|^2+c_s,
\end{equation}
where
\begin{align*}
a_s&:=1-2c_1s+2L_xs+\sqrt2K_Gs
           +16K^2s^2\theta^{-2\alpha}+6K_G^2s^2,\\
b_s&:=\sqrt2K_Gs+6K_G^2s^2,\\
c_s&:=(2c_2+\sqrt2K_G)s
           +16K^2s^2\theta^{-2\alpha}+6K_G^2s^2.
\end{align*}
To see this, writing
$g_i(x):=(N-1)^{-1}\sum_{j\ne i}G(x^i,x^j)$, one has
\[
|b_\alpha(x^i)+g_i(x)|^2
\le16K^2\theta^{-2\alpha}(1+|x^i|^2)
+6K_G^2\left(1+|x^i|^2+\frac1{N-1}\sum_{j\ne i}|x^j|^2\right)
\]
and
\[
2\langle x^i,g_i(x)\rangle
\le(2L_x+\sqrt2K_G)|x^i|^2
+\sqrt2K_G\left(1+\frac1{N-1}\sum_{j\ne i}|x^j|^2\right).
\]
The latter inequality follows by splitting
$G(x^i,x^j)=G(x^i,x^j)-G(0,x^j)+G(0,x^j)$ and applying
Young's inequality with parameter $\sqrt2K_G$.

The definition of $\widetilde\eta$ ensures that $a_s,b_s\ge0$ and
\begin{equation}\label{eq:particle-substep-coefficients}
a_s+b_s\le1-\widetilde\kappa_\alpha s,
\qquad c_s\le D_{\rm P}s,
\qquad 0\le s\le\theta.
\end{equation}
Indeed, $a_s\ge1-(2c_1-2L_x-\sqrt2K_G)s\ge0$, and
\[
a_s+b_s
\le1-2(c_1-L_x-\sqrt2K_G)s
 +\bigl(16K^2\theta^{1-2\alpha}+12K_G^2\theta\bigr)s.
\]
For $\alpha<1/2$, the term in parentheses is bounded above by
$(16K^2+12K_G^2)\theta^{1-2\alpha}
 \le c_1-L_x-\sqrt2K_G$.
For $\alpha=1/2$, use
$12K_G^2\theta\le c_1-L_x-\sqrt2K_G-8K^2$.
These are exactly restrictions already included in $\widetilde\eta$.
Also, $s^2\theta^{-2\alpha}\le s$ and $s^2\le s$, proving the
bound for $c_s$. In particular,
\begin{equation}\label{wr1}
\widetilde r_1:=1-\widetilde\kappa_\alpha\theta\in(0,1).
\end{equation}

Set
\[
U_r^N(t):=\bigl(\mathbb E|\widetilde X_t^{1,N}|^r\bigr)^{2/r},
\qquad
H_{r,d}^{\rm P}:=
\frac{D_{\rm P}+2(d+r-2)}{\widetilde\kappa_\alpha}.
\]
Taking $L^{r/2}$ norms in \eqref{eq:particle-deterministic-substep},
Minkowski's inequality and exchangeability yield
\begin{align*}
(\mathbb{E}|\Phi_{\theta,s}^1(\widetilde X_{k\theta}^N)|^r)^{\frac2r}
&\le a_sU_r^N(k\theta)
 +\frac{b_s}{N-1}\sum_{j\ne1}
       \bigl(\mathbb E|\widetilde X_{k\theta}^{j,N}|^r\bigr)^{2/r}
 +c_s\\
&=(a_s+b_s)U_r^N(k\theta)+c_s\\
&\le(1-\widetilde\kappa_\alpha s)U_r^N(k\theta)+D_{\rm P}s.
\end{align*}
By \eqref{interpolation} and
Lemma~\ref{lem:gaussian-moment-explicit}, it follows that
\begin{equation}\label{eq:particle-root-moment-recursion}
U_r^N(k\theta+s)
\le(1-\widetilde\kappa_\alpha s)U_r^N(k\theta)
  +\bigl(D_{\rm P}+2(d+r-2)\bigr)s.
\end{equation}
Taking $s=\theta$ and iterating in $k$, we obtain
\[
U_r^N(k\theta)
\le\widetilde r_1^k m_r^{2/r}
 +(1-\widetilde r_1^k)H_{r,d}^{\rm P}
\le\max\{m_r^{2/r},H_{r,d}^{\rm P}\}.
\]
Using \eqref{eq:particle-root-moment-recursion} for
$0\le s\le\theta$ gives the same bound between grid points.
All constants are independent of $N$, $k$, and $\theta$ in the stated
range. This proves \eqref{eq:moment-particle-explicit}.
The bounds for $0<r<2$ follow from Lyapunov's inequality.
\end{proof}

\begin{proof}[Proof of Lemma \ref{primu2}]
	By the definition of the interpolation processes \eqref{interpolation}, \eqref{prop1}, \eqref{Glip}, \eqref{exchan} and Lemma \ref{primul}, for any $t\in[k\theta,(k+1)\theta],$
	\begin{align*}
		&\sum_{i=1}^N\mathbb{E}[|\widetilde{X}_t^{i,N}-\widetilde{X}_{k\theta}^{i,N}|^2]\\
		=&\sum_{i=1}^N\mathbb{E}[|(t-k\theta)b_\alpha(\widetilde{X}^{i,N}_{k\theta})+\frac{t-k\theta}{N-1}\sum\limits_{j\neq i}G(\widetilde{X}^{i,N}_{k\theta},\widetilde{X}_{k\theta}^{j,N})+\sqrt{2}(B_t^i-B^i_{k\theta})|^2]\\
		=&\sum_{i=1}^N\mathbb{E}[(t-k\theta)^2|b_\alpha(\widetilde{X}^{i,N}_{k\theta})+\frac{1}{N-1}\sum\limits_{j\neq i}G(\widetilde{X}^{i,N}_{k\theta},\widetilde{X}_{k\theta}^{j,N})|^2]+2(t-k\theta)Nd\\
		\leq&2\theta Nd+2\theta^2\sum_{i=1}^N\mathbb{E}[|b_\alpha(\widetilde{X}_{k\theta}^{i,N})|^2+|\frac{1}{N-1}\sum\limits_{j\neq i}G(\widetilde{X}_{k\theta}^{i,N},\widetilde{X}_{k\theta}^{j,N})|^2]\\
		\leq&2\theta Nd+\mathbb{E}\sum_{i=1}^N[16K^2 \theta^{2-2\alpha}(1+|\widetilde{X}_{k\theta}^{i,N}|^2)+\frac{6\theta^2K_G^2}{N-1}\sum\limits_{j\neq i}(1+|\widetilde{X}_{k\theta}^{i,N}|^2+|\widetilde{X}_{k\theta}^{j,N}|^2)]\\
		\leq&(2\theta Nd+16K^2 \theta^{2-2\alpha}N+6\theta^2K_G^2N)+(16K^2 \theta^{2-2\alpha}+12\theta^2K_G^2)\mathbb{E}[\sum_{i=1}^N|\widetilde{X}^{i,N}_{k\theta}|^2]\\
		\leq&(2\theta Nd+16K^2 \theta^{2-2\alpha}N+6\theta^2K_G^2N)+(16K^2 \theta^{2-2\alpha}+12\theta^2K_G^2)N\widetilde{C}_2\\
		\leq&N\theta(2d+16K^2+6K_G^2 +16K^2\widetilde{C}_2+12K_G^2\widetilde{C}_2).
	\end{align*}
	This completes the proof.
\end{proof}
\subsection{\texorpdfstring{\blue{Moment estimates for the McKean--Vlasov equation \eqref{mckeanvlasov}}}{Moment estimates for the McKean--Vlasov equation}}
\begin{lem}\label{lem:MVmom}
	Suppose that $G\in C^1(\mathbb R^d\times\mathbb R^d;\mathbb R^d)$ has bounded partial derivatives and that there exist constants $c_1,c_2>0$ such that
	\[
	\langle x,b(x)\rangle\le -c_1|x|^2+c_2.
	\]
	Set
	\[
	\qquad \lambda_G:=c_1-\|\nabla_xG\|_\infty-\|\nabla_yG\|_\infty.
	\]
	Let $\bar X_t$ be a solution to \eqref{mckeanvlasov} and assume $\lambda_G>0$. If $\mathbb E|\bar X_0|^2<\infty$, then
	\[
	\sup_{t\ge0}\mathbb E|\bar X_t|^2<\infty.
	\]
	If additionally $\mathbb E|\bar X_0|^4<\infty$, then
	\[
	\sup_{t\ge0}\mathbb E|\bar X_t|^4<\infty.
	\]
\end{lem}
\begin{proof}
	Write $g_0:=|G(0,0)|$. The bounded derivatives of $G$ give
	\[
	|G(x,y)|\le g_0+\|\nabla_xG\|_\infty|x|+\|\nabla_yG\|_\infty|y|.
	\]
	\blue{Set $M_r(t):=\mathbb E|\bar X_t|^r$. We apply It\^o's formula to $M_2(t)$ and $M_4(t)$; the standard stopping-time localization argument justifies both applications.} 
    For the second moment,
	\begin{align*}
	M_2'(t)
	&\le -2c_1M_2(t)+2c_2+2d+
	  2\mathbb E\big\langle\bar X_t,G\circledast\mu_{\bar X_t}(\bar X_t)\big\rangle\\
	&\le -2\lambda_GM_2(t)+2g_0M_2(t)^{1/2}+2c_2+2d\\
	&\le -\lambda_GM_2(t)+2c_2+2d+\frac{g_0^2}{\lambda_G}.
	\end{align*}
	Gronwall's inequality therefore yields $\sup_{t\ge0}M_2(t)<\infty$.

	For the fourth moment, It\^o's formula and H\"older's inequality give
	\begin{align*}
	M_4'(t)
	\le{}&-4c_1M_4(t)+(4c_2+4d+8)M_2(t)\\
	&+4\mathbb E\left[|\bar X_t|^2
	\big\langle\bar X_t,G\circledast\mu_{\bar X_t}(\bar X_t)\big\rangle\right]\\
	\le{}&-4\lambda_GM_4(t)+(4c_2+4d+8)M_2(t)
	+4g_0\mathbb E|\bar X_t|^3,
	\end{align*}
	where we used
	\[
	\mathbb E|\bar X_t|^3\,\mathbb E|\bar X_t|
	\le M_4(t)^{3/4}M_4(t)^{1/4}=M_4(t).
	\]
	For every $\delta>0$, Young's inequality yields
	\[
	g_0\mathbb E|\bar X_t|^3
	\le \delta M_4(t)+\frac{27g_0^4}{256\delta^3}.
	\]
	Taking $\delta=\lambda_G/2$, we obtain
	\[
	M_4'(t)\le-2\lambda_GM_4(t)
	+(4c_2+4d+8)M_2(t)+\frac{27g_0^4}{8\lambda_G^3}.
	\]
	The uniform second-moment bound and Gronwall's inequality now imply 
\[
\sup_{t\ge0}M_4(t)<\infty.
\]
\end{proof}

\vskip0.5cm
\noindent{\bf Acknowledgments }
The authors would like to thank Prof.~Liming Wu for his valuable comments and suggestions on this manuscript. The research of the fourth named author is supported by NNSFC Grant 12471143.

\bibliographystyle{amsplain}
\bibliography{euler-mv-references}

\end{document}